\documentclass{amsart}

\usepackage[utf8]{inputenc}
\usepackage{amssymb}
\usepackage{amsmath}
\usepackage{amsthm}
\usepackage{graphicx}
\usepackage{cleveref}
\usepackage{color}
\usepackage{enumerate}
\usepackage[foot]{amsaddr}
\usepackage[T1]{fontenc}
\usepackage[shortlabels]{enumitem}
\usepackage{url}

\newtheorem{thrm}{Theorem}
\newtheorem{crllr}[thrm]{Corollary}
\newtheorem{lmm}[thrm]{Lemma}

\newtheorem{bsrvtn}[thrm]{Observation}
\newtheorem{prpstn}[thrm]{Proposition}
\newtheorem{dfntn}[thrm]{Definition}

\theoremstyle{remark}
\newtheorem{rmrk}[thrm]{Remark}
\newtheorem{lgrthm}[thrm]{Algorithm}
\newtheorem{xmpl}[thrm]{Example}

\def\N{\mathbb N}
\def\A{\mathcal A}
\def\B{\mathcal B}

\def\S{\mathcal S}
\def\uu{\mathbf u}
\def\vv{\mathbf v}
\def\dd{\mathbf d}
\def\zz{\mathbf z}

\def\ff{\mathbf f}
\def\gg{\mathbf g}

\begin{document}

\title{Derivated sequences of complementary symmetric Rote sequences}

\author{Kateřina Medková} 
\author{Edita Pelantová} 
\author{Laurent Vuillon}

\address[K. Medková, E. Pelantová]{Faculty of Nuclear Sciences and Physical Engineering, Czech Technical Univ. in Prague, Trojanova 13, 120 00, Prague 2, Czech Republic}

\address[L. Vuillon]{Univ. Grenoble Alpes, Univ. Savoie Mont Blanc, CNRS, LAMA, 73000 Chamb\'ery, France}

\email{katerinamedkova@gmail.com, Edita.Pelantova@fjfi.cvut.cz, 
Laurent.Vuillon@univ-smb.fr}

\date{\today}

\begin{abstract}
Complementary symmetric Rote sequences are binary sequences which have factor complexity $\mathcal{C}(n) = 2n$ for all integers $n \geq 1$ and whose languages are closed under the exchange of letters.   
These sequences are intimately linked to Sturmian sequences. 
Using this connection we investigate the return words and the derivated sequences to the prefixes of any complementary symmetric Rote sequence $\bf v$ which is associated with a standard Sturmian sequence $\bf u$. 
We show that any non-empty prefix of $\bf v$ has three return words.
We prove that any derivated sequence of $\bf v$ is coding of three interval exchange transformation and we determine the parameters of this transformation. 
We also prove that ${\bf v}$ is primitive substitutive if and only if $\bf u$ is primitive substitutive. 
Moreover, if the sequence $\bf u$ is a fixed point of a primitive morphism, then all derivated sequences of $\bf v$ are also fixed by primitive morphisms.  
In that case we provide an algorithm for finding these fixing morphisms.   
\medskip


\noindent \textit{2000MSC: } 68R15
\end{abstract}

\maketitle

\section{Introduction}
The notion of return words and derivated sequences has been introduced by Durand in \cite{Dur98} and seems to be a powerful tool for studying the structure of aperiodic infinite sequences, and so also of the corresponding dynamical systems. 

A return word can be considered as a symbolical analogy of return time occurring in the theory of dynamical systems.
Let $\uu = u_0u_1u_2 \cdots$ be a sequence and let $w = u_iu_{i+1} \cdots u_{i+n-1}$ be its factor. 
The index $i$ is an occurrence of $w$. 
A return word to $w$ is a word $u_iu_{i+1} \cdots u_{j-1}$  with $i < j$ being two consecutive occurrences of $w$. 

Return words are well understood in the case of Sturmian sequences, i.e. aperiodic sequences with the lowest possible factor complexity $C(n) = n+1$ for all $n \in \N$.  
They can be also seen as the coding of rotation with an irrational angle $\alpha$ on the unit circle with the partition in two intervals of lengths $\alpha$ and $1- \alpha$, respectively. 

The third author characterizes Sturmian sequences as sequences with two return words to each their factor in \cite{Vui01}. 
Similarly the paper \cite{BaPeSt08} is dedicated to investigation of  sequences with a fixed number of return words to any factors, in particular, Arnoux-Rauzy sequences and sequences coding interval exchange transformations are of this type, see \cite{Vui00}.
Besides, return words in episturmian sequences were described in \cite{JuVui00} while the description of return words in the coding of rotations was used to show their fullness in \cite{BlBrLaVui09}. 

A derivated sequence expresses the order of return words in the sequence $\uu$. 
More precisely, if $w$ is a prefix of $\uu$ with $k$ return words $r_1$, $r_2, \ldots, r_k$, then $\uu$ can be written as the concatenation of these return words: $\uu = r_{i_1}r_{i_2}r_{i_3} \cdots$.
Then the derivated sequence of $\uu$ to the prefix $w$ is the sequence $\dd = i_1i_2i_3 \cdots$ over an alphabet of size $k$.  

Durand's result from \cite{Dur98} states that a sequence is primitive substitutive if and only if its number of derivated sequences is finite. 
Now the goal is to understand the structure of the derivated sequences. 
Derivated sequences of standard Sturmian sequences were investigated in \cite{ArBr} and the derivated sequences of fixed points of primitive Sturmian morphisms were described in \cite{KlMePeSt18}.

Recently, new developments are done to understand the structure of more complicated objects, e.g. acyclic, neutral and tree sequences introduced in \cite{6peopleDEF}. 
Return words in sequences coding linear involutions were studied in \cite{6people} and the number of return words for more general neutral sequences was determined in \cite{DoPe}. 
In \cite{4people+Durand} the properties of return words and derivated sequences were exploited for the characterization of substitutive tree sequences.

In this paper, we study complementary symmetric Rote sequences, i.e. the binary sequences with factor complexity $C(n)= 2n$ for all $n \geq 1$ whose languages are closed under the exchange of letters. 
These sequences are not tree, but they represent an interesting example of neutral sequences with characteristics $1$.
They are named after G. Rote, who proposed several constructions of these sequences in \cite{Ro94}. 
For example, he constructed them as projections of fixed points over a four letter alphabet (see Section \ref{otherview} of our paper) or as the coding of irrational rotations on a unit circle with the partition on two intervals of length $1/2$. 
Later on, they were also constructed using palindromic and pseudopalindromic closures, see \cite{BlPaTrVui13}.      
This construction was proposed and firstly applied to the Thue-Morse sequence in \cite{deDe} and later extended in \cite{PeSt16} to a broader class of sequences.

Our techniques are based on the close relation between complementary symmetric Rote sequences and Sturmian sequences shown in \cite{Ro94}: a sequence $\vv = v_0v_1v_2$ is a complementary symmetric Rote sequence if and only if its difference sequence $\uu$, which is defined by $u_i = v_{i+1}- v_i \mod 2$, is a Sturmian sequence.  
In fact, we investigate the consequences of this relation, see Sections \ref{S_Rote}, \ref{S_ReturnWords} and \ref{types}. 
We also use the description of derivated sequences of Sturmian sequences as studied in detail in \cite{KlMePeSt18}. 
Here we focus on complementary symmetric Rote sequences which are associated with standard Sturmian sequences. 

First we recall needed terminology and notations in Section \ref{S_Preliminaries}.
Section \ref{S_ReturnWords} is dedicated to return words: in Theorem \ref{Thm_ReturnWordsNumber} we show that every non-empty prefix of any complementary symmetric Rote sequence $\vv$ has three return words.  
In other words, all derivated sequences of $\vv$ are over a ternary alphabet.
Then we proceed with the study of derivated sequences.  
In Proposition \ref{Prop_ThreeExchange} we prove that any derivated sequence of $\vv$ is the coding of a three interval exchange transformation and we determine the parameters of this transformation.
Then in Theorem \ref{Substitutive} and Lemma \ref{notFixed} we concentrate on the question of substitutivity of Rote sequences.  
In the case when the associated standard Sturmian sequence ${\bf u}$ is fixed by a primitive morphism, Corollary \ref{Thm_Number} estimates the number of distinct derivated sequences of $\vv$ from above and Algorithm \ref{Algo_FixingMorphism} provides a list of all derivated sequences of $\vv$.  
Section \ref{otherview} compares our techniques with the original Rote's construction of substitutive Rote sequences and the last section collects related open questions.

\section{Preliminaries} \label{S_Preliminaries}

\subsection{Sequences and morphisms}

An \textit{alphabet} $\A$ is a finite set of symbols called \textit{letters}.
A \textit{word} over $\A$ of \textit{length} $n$ is a string $u = u_0u_1 \cdots u_{n-1}$, where $u_i \in \A$ for all $i \in \{0,1, \ldots, n-1\}$. The length of $u$ is denoted by $|u|$.
The set of all finite words over $\A$ together with the operation of concatenation form a monoid $\A^*$.
Its neutral element is the \textit{empty word} $\varepsilon$ and we denote $\A^+ = \A^* \setminus \{\varepsilon\}$.

If $u = xyz$ for some $x,y,z \in \A^*$, then $x$ is a \textit{prefix} of $u$, $z$ is a \textit{suffix} of $u$ and $y$ is a \textit{factor} of $u$.

To any word $u$ over $\A$ with the cardinality $\#\A = d$  we assign the vector $V_u \in \N^{d}$ defined as $(V_u)_a = |u|_a$ for all $a \in \A$, where $|u|_a$ is the number of letters $a$ occurring in $u$. 
The vector $V_u$ is usually called the \textit{Parikh vector} of $u$.

A \textit{sequence} over $\A$ is an infinite string $\uu = u_0u_1u_2 \cdots$, where $u_i \in \A$ for all $i \in \N = \{0,1,2, \ldots\}$.
We always denote sequences by bold letters.
The set of all sequences over $\A$ is denoted $\A^\N$.
A sequence $\uu$ is \textit{eventually periodic} if $\uu = vwww \cdots = v(w)^\infty$ for some $v \in \A^*$ and $w \in \A^+$, moreover, $\uu$ is \textit{purely periodic} if $\uu = www \cdots = w^\infty$.
Otherwise $\uu$ is \textit{aperiodic}.

A \textit{factor} of $\uu$ is a word $y$ such that $y = u_iu_{i+1}u_{i+2} \cdots u_{j-1}$ for some $i, j \in \N$, $i \leq j$.
The index $i$ is called an \textit{occurrence} of the factor $y$ in $\uu$.
In particular, if $i = j$, the factor $y$ is the empty word $\varepsilon$ and any index $i$ is its occurrence.
If $i=0$, the factor $y$ is a \textit{prefix} of $\uu$.

If each factor of $\uu$ has infinitely many occurrences in $\uu$, the sequence $\uu$ is \textit{recurrent}.
Moreover, if the distances between two consecutive occurrences are bounded, $\uu$ is \textit{uniformly recurrent}.

The \textit{language} $\mathcal{L}(\uu)$ of the sequence $\uu$ is the set of all factors of $\uu$.
A factor $w$ of $\uu$ is \textit{right special} if both words $wa$ and $wb$ are factors of $\uu$ for at least two distinct letters $a,b \in \A$.
Analogously we define the \textit{left special} factor.
The factor is \textit{bispecial} if it is both left and right special.
Note that the empty word $\varepsilon$ is the bispecial factor if at least two distinct letters occur in $\uu$.

The \textit{factor complexity} of a sequence $\uu$ is the mapping $\mathcal{C}_\uu: \N \to \N$ defined by
$$\mathcal{C}_\uu(n) = \# \{w \in \mathcal{L}(\uu) : |w| =  n \}\, .$$
A classical result of Hedlund and Morse \cite{HeMo} says that a sequence is eventually periodic if and only if its factor complexity is bounded.
The factor complexity of any aperiodic sequence $\uu$ satisfies $\mathcal{C}_\uu(n) \geq n+1$ for every $n \in \N$.

A \textit{morphism} from a monoid $\A^*$ to a monoid $\B^*$ is a mapping $\psi: \A^* \to \B^*$ such that $\psi(uv) = \psi(u)\psi(v)$  for all $u, v \in \A^*$.
In particular, if $\A = \B$, $\psi$ is a morphism over $\A$.
The domain of a morphism $\psi$ can be naturally extended to $\A^\N$ by
$$\psi(\uu)=\psi(u_0u_1u_2 \cdots) = \psi(u_0)\psi(u_1)\psi(u_2) \cdots\,. $$

The \textit{matrix} of a morphism $\psi$ over $\A$ with the cardinality $\#\A = d$ is the matrix $M_\psi \in \N^{d\times d}$ defined as $(M_\psi)_{ab} = |\psi(a)|_b$ for all $a,b \in \A$.
The Parikh vector of the $\psi$-image of a word $w\in \A^*$ can be obtained via multiplication by the matrix $M_\psi$, i.e.
$V_{\psi(w)} = M_\psi V_w$.

The morphism is \textit{primitive} if there is a positive integer $k$ such that all elements of the matrix $M_\psi^k$ are positive.
A \textit{fixed point} of a morphism $\psi$ is a sequence $\uu$ such that $\psi(\uu) = \uu$.
It is well known that all fixed points of a primitive morphism have the same language.
The sequence $\uu \in \A^\N$ is \textit{primitive substitutive} if $\uu = \sigma(\vv)$ for a morphism $\sigma: \B^* \to \A^*$ and a sequence $\vv \in \B^\N$ which is a fixed point of a primitive morphism over $\B$.

\subsection{Derivated sequences}

Consider a prefix $w$ of a recurrent sequence $\uu$.
Let $i < j$ be two consecutive occurrences of $w$ in $\uu$.
Then the word $u_iu_{i+1} \cdots u_{j-1}$ is a \textit{return word} to $w$ in $\uu$. The set of all return words to $w$ in $\uu$ is denoted $\mathcal{R}_\uu(w)$.

If the sequence $\uu$ is uniformly recurrent, the set $\mathcal{R}_\uu(w)$ is finite for each prefix $w$, i.e. $\mathcal{R}_\uu(w) = \{r_0, r_1, \ldots, r_{k-1}\}$.
Then the sequence $\uu$ can be written as a concatenation of these return words:
$$\uu = r_{d_0}r_{d_1}r_{d_2} \cdots$$
and the \textit{derivated sequence} of $\uu$ to the prefix $w$ is the sequence $\dd_\uu(w) = d_0d_1d_2 \cdots$ over the alphabet of cardinality $\# \mathcal{R}_\uu(w) = k$.
For simplicity, we do not fix this alphabet and we consider two derivated sequences which differ only in a permutation of letters as identical.
The set of all derivated sequences to the prefixes of $\uu$ is
$$\text{Der}(\uu) = \{\dd_\uu(w) : w \ \text{is a prefix of} \ \uu\}\,.$$

If the prefix $w$ is not right special, there is a unique letter $a$ such that $wa$ is a factor of $\uu$.
It means that the occurrences of factors $w$ and $wa$ in $\uu$ coincides, thus $\mathcal{R}_\uu(w) = \mathcal{R}_\uu(wa)$ and $\dd_\uu(w) = \dd_\uu(wa)$.
If $\uu$ is aperiodic, then any prefix of $\uu$ is a prefix of some right special prefix of $\uu$.
Therefore, for an aperiodic uniformly recurrent sequence $\uu$ we can take into consideration only right special prefixes since
\begin{equation}\label{onlyRight}
\text{Der}(\uu) = \{\dd_\uu(w) : w \ \text{is a right special prefix of} \ \uu\}\,.
\end{equation}

In the sequel we will essentially use the following Durand's result.

\begin{thrm}[Durand \cite{Dur98}]\label{Durand}
A sequence $\uu$ is substitutive primitive if and only if the set $\text{\rm Der}(\uu)$ is finite.
\end{thrm}

\subsection{Sturmian sequences} \label{S_SturmianSequences}

Sturmian sequences are aperiodic sequences with the lowest possible factor complexity.
In other words, a sequence $\uu$ is \textit{Sturmian} if it has its factor complexity $\mathcal{C}_\uu(n) = n+1$ for all $n \in \N$.
Clearly, all Sturmian sequences are defined over a binary alphabet.

There are many equivalent definitions of Sturmian sequences, see for example  \cite{Be, BeSe, BaPeSt}.
One of the most important characterizations of Sturmian sequences comes from the symbolic dynamics: any Sturmian sequence can be obtained by a coding of two interval exchange transformation.
Here we recall only the basic facts about this transformation, a detailed explanation can be found in  \cite{Lo83}.

For a given parameter $\alpha \in (0,1)$, consider the partition of the interval $I=[0,1)$ into $I_0=[0,\alpha)$ and $I_1=[\alpha, 1)$ or the partition of $I=(0,1]$ into $I_0=(0,\alpha]$ and   $I_1=(\alpha, 1]$.
Then the \textit{two interval exchange transformation} $T: I\to I$ is defined by
$$
T(y) =
\left\{
\begin{aligned}
&y + 1 - \alpha \ \ \ \text{if} \ y \in I_0\,,\\
&y -  \alpha \ \ \ \ \  \ \ \ \text{if} \ y \in I_1\,.
\end{aligned}
\right.
$$
If we take an initial point $ \rho \in I$, the sequence $\uu = u_0u_1u_2 \cdots \in \{0,1\}^\N$ defined by
$$
u_n =
\left\{
\begin{aligned}
& 0 \ \ \ \ \ \ \ \text{if} \ T^n(\rho) \in  I_0\,,\\
&1 \ \ \ \ \ \ \  \text{if} \ T^n(\rho)  \in \ I_1\,
\end{aligned}
\right.
$$
is a \textit{2iet sequence} with the \textit{slope} $\alpha$ and the \textit{intercept} $\rho$.
It is well known that the set of all 2iet sequences with irrational slopes coincides with the set of all Sturmian sequences.

Any Sturmian sequence is uniformly recurrent.
The language of a Sturmian sequence is independent of its intercept $\rho$, i.e. it depends only on its slope $\alpha$.
The frequencies of the letters $0$ and $1$ in a Sturmian sequence with the slope $\alpha$ are $\alpha$ and $1-\alpha$,  respectively.
In the case that $\alpha > \tfrac12$, the form of the transformation $T$ implies that two consecutive occurrences of the letter $1$ are separated by the block $0^k$ or $ 0^{k+1}$, where $k=\lfloor \tfrac{\alpha}{1-\alpha}\rfloor$.
Similarly, if $\alpha < \tfrac12$, two $0$'s are separated by the block $1^k$ or $ 1^{k+1}$, where $k=\lfloor \tfrac{1-\alpha}{\alpha}\rfloor$.

Among all Sturmian sequences with a given slope $\alpha$, the sequence with the intercept $\rho =1-\alpha$ plays a special role.
Such a sequence is called a \textit{standard Sturmian sequence} and it is usually denoted by  ${\bf c}_\alpha$.
Any prefix of ${\bf c}_\alpha$ is a left special factor.
In other words, a Sturmian sequence $\uu \in \{0,1\}^\N$ is standard if both sequences $0\uu$, $1\uu$ are Sturmian.
In particular, it means that
\begin{itemize}
\item if $\alpha >\tfrac12$, then ${\bf c}_\alpha$ has a prefix $0^k1$ and ${\bf c}_\alpha$ can be uniquely written as a concatenation of the blocks $0^k1$ and $0^{k+1}1$;
\item if $\alpha <\tfrac12$, then ${\bf c}_\alpha$ has a prefix $1^k0$ and ${\bf c}_\alpha$ can be uniquely written as a concatenation of the blocks $1^k0$ and $1^{k+1}0$.
\end{itemize}
Moreover, a factor of ${\bf c}_\alpha$ is bispecial if and only if it is a palindromic prefix of ${\bf c}_\alpha$.

In the context of derivated sequences the most important characterization of Sturmian sequences is provided by the third author in \cite{Vui01}: a given sequence $\uu$ is Sturmian if and only if any prefix of $\uu$ has exactly two return words.

Let us denote the return words to a prefix $w$ of a Sturmian sequence $\uu$  as $\mathcal{R}_\uu(w) = \{r, s\}$.
Then the derivated sequence $\dd_\uu(w)$ of $\uu$ to the prefix $w$ can be considered over the alphabet $\{r, s\}$, i.e. $\dd_\uu(w) \in \{r, s\}^\N$.
As follows from Vuillon's \cite{Vui01} and Durand's result \cite{Dur98}, the derivated sequence $\dd_\uu(w)$ of a Sturmian sequence $\uu$ is also a Sturmian sequence.
Moreover, if $\uu$ is standard, then both sequences $0\uu$ and $1\uu$ are Sturmian.
It implies that $r\dd_\uu(w)$ and $s\dd_\uu(w)$ are Sturmian as well.
We can conclude that the derivated sequence to any prefix of a standard Sturmian sequence is a standard Sturmian sequence.
It also means that $\dd_\uu(w) \in \{r, s\}^\N$ can be decomposed into blocks $r^ks$ and $r^{k+1}s$, where $k$ is a positive integer and $r$ is the most frequent return word.
We will strictly use this notation through the whole paper.

In \cite{ArBr}, Ara\'ujo and Bruy\`ere described derivated sequences of any standard Sturmian sequence $\uu$.
Their description uses the continued fraction of the slope $\alpha$ of $\uu$.
Derivated sequences of all Sturmian sequences are studied in \cite{KlMePeSt18}.
In the sequel,  we will work only with standard Sturmian sequences since especially in this case the elements of the set $\text{Der}(\uu)$ are easily expressible.
In accordance with a wording  provided in \cite{KlMePeSt18}, the S-adic representation of $\uu$ by a sequence of Sturmian morphisms will be used for expression of  the set $\text{Der}(\uu)$.

\subsection{ Sturmian morphisms}

A morphism $\psi:\{0,1\}^* \to \{0,1\}^* $ is \textit{Sturmian} if $\psi(\uu)$ is a Sturmian sequence for any Sturmian sequence $\uu$.
The set of all Sturmian morphisms together with the operation of composition form the so-called Sturmian monoid $St$.
This monoid is generated by two morphisms $E$ and $F$, where $E$ is the morphism which exchanges letters, i.e. $E: 0 \to 1$, $1 \to 0$, and $F$ is the Fibonacci morphism, i.e. $F: 0 \to 01$, $1 \to 0$.
In the sequel, we work with the submonoid of $St$ which is   generated by two elementary morphisms $\varphi_b$ and $\varphi_\beta$ defined by
$$
\varphi_b = F\circ E: \quad
\left\{ \begin{aligned}
0 &\to 0 \\
1 &\to 01\,
\end{aligned} \right.
 \ \ \ \ \ \ \ \text{and}  \ \ \ \ \ \ \ \ \ \
\varphi_\beta = E\circ F: \quad
\left\{ \begin{aligned}
0 &\to 10 \\
1 &\to 1\,
\end{aligned}\right. .
$$
Their corresponding matrices are:
$$
M_b = \left(\begin{array}{cc}
1 & 1\\
0 & 1\\
\end{array} \right)\quad \quad \ \text{and} \quad \quad \quad M_\beta = \left(\begin{array}{cc}
1 & 0\\
1 & 1\\
\end{array} \right)\,.
$$

The image of a standard Sturmian sequence under $\varphi_b$ or $\varphi_\beta$ is a standard Sturmian sequence as well.
Therefore, any element of the submonoid $\left< \varphi_b, \varphi_\beta\right>$ preserves the set of standard Sturmian sequences.
For some $z = z_0z_1 \cdots z_{n-1} \in \{b, \beta\}^+$, the composition of the morphisms $\varphi_{z_0}, \varphi_{z_1}, \varphi_{z_2}, \ldots , \varphi_{z_{n-1}}$ will be denoted by $\varphi_z =  \varphi_{z_0}\varphi_{z_1} \cdots \varphi_{z_{n-1}} $.
Let us stress that the morphism $\varphi_z$ is primitive if and only if $z$ contains both letters $b$ and $\beta$.
By $\varphi_\varepsilon$ we denote the identity morphism.

\begin{lmm} \label{Lem_Standard}
For every standard Sturmian sequence $\uu$ there is a uniquely given standard Sturmian sequence $\uu'$ such that $\uu = \varphi_b(\uu')$ or $\uu = \varphi_\beta(\uu')$.
\end{lmm}

\begin{proof}
Let us suppose that the letter $0$ is more frequent in $\uu$ (the second case can be proved analogously).
Since $\uu$ is a standard Sturmian sequence, it can be written as a concatenation of blocks $0^k1$ and $0^{k+1}1$ for some integer $k \geq 1$.
Thus $\uu$ can be uniquely desubstituted by $0 \to 0$ and $01 \to 1$ to the standard Sturmian sequence $\uu'$ which is a concatenation of blocks $0^{k-1}1$ and $0^{k}1$.
Therefore $\uu = \varphi_b(\uu')$.
\end{proof}

By the previous lemma, to a given standard Sturmian sequence $\uu$ we can uniquely assign the pair: the \textit{directive sequence} $\zz = z_0z_1 \cdots \in \{b, \beta\}^\N$ and the sequence $(\uu^{(n)})_{n \geq 0}$, such that
$$
 \uu^{(n)} \in \{0,1\}^\N \ \ \text{ is a standard Sturmian sequence \ and } \
\uu = \varphi_{z_0z_1 \ldots z_{n-1}}(\uu^{(n)})\,   \ \text{for every } \ n \in \N\,.
$$
In fact, a sequence $\zz \in \{b, \beta\}^\N$ containing infinitely many occurrences of both letters already determines a unique standard Sturmian sequence $\uu$, as
$$\uu = \lim_{n\to \infty}     \varphi_{z_0z_1 \ldots z_{n-1}}(0) = \lim_{n\to \infty}     \varphi_{z_0z_1 \ldots z_{n-1}}(1)\,.$$

Now we can formulate several simple consequences of Lemma \ref{Lem_Standard}. 

\begin{bsrvtn} \label{Obs_Epsilon}
Let $\uu$ be a standard Sturmian sequence with the directive sequence $\zz \in \{b, \beta\}^\N$.
\begin{itemize}
\item[i)] The sequence $\zz$ contains infinitely many letters $b$ and infinitely many letters $\beta $.
\item[ii)] If $\zz$ has a prefix  $b^k \beta$ for some positive integer $k$, then the letter $0$ is more frequent in $\uu$ and $\uu$ can be written as a concatenation of blocks $0^k1$ and $0^{k+1}1$.
\item[iii)] If $\zz$ has a prefix $\beta^k b$ for some positive integer $k$, then the letter $1$ is more frequent in $\uu$ and $\uu$ can be written as a concatenation of blocks $1^k 0$ and $1^{k+1}0$.
\item[iv)] The directive sequence $\zz$ is eventually periodic if and only if the sequence $\uu$ is substitutive.
Moreover, $\zz$ is purely periodic, i.e. $\zz = z^\infty$ for some $z \in \{b, \beta\}^+$, if and only if $\uu$ is a fixed point of the morphism $\varphi_z$.
\end{itemize}
\end{bsrvtn}

\subsection{Complementary symmetric Rote sequences} \label{S_Rote}

A \textit{Rote sequence} is a sequence $\vv$ with the factor complexity $\mathcal{C}_\vv(n) = 2n$ for all integer $n \geq 1$.
Clearly, all Rote sequences are defined over a binary alphabet, e.g. $\{0,1\}$.
If the language of  a Rote sequence $\vv$ is  closed under the exchange of letters, i.e. $E(v)\in \mathcal{L}(\vv)$ for each $v \in \mathcal{L}(\vv)$, the Rote sequence  $\vv$  is called \textit{complementary symmetric}.
Rote in \cite{Ro94} proved that these sequences are essentially connected with Sturmian sequences:

\begin{prpstn}[Rote \cite{Ro94}] \label{Prop_SturmRote}
Let $\uu = u_0u_1 \cdots $ and $\vv = v_0v_1 \cdots$ be two sequences over $\{0,1\}$ such that $u_i = v_i + v_{i+1}\mod 2$ for all $i \in \N$. Then $\vv$ is a complementary symmetric Rote sequence if and only if $\uu$ is a Sturmian sequence.
\end{prpstn}

\medskip

\noindent {\bf Convention. }In this paper, we work only with complementary symmetric Rote sequences and for simplicity we usually call them shortly Rote sequences.
\medskip

As indicated by Proposition \ref{Prop_SturmRote}, it will be useful to introduce the following notation.

\begin{dfntn}\label{S}  By $\S$ we denote the mapping $\S: \{0,1\}^+ \to \{0,1\}^*$ such that for every $v_0\in \{0,1\}$ we put $\S(v_0) = \varepsilon$ and for every $v = v_0v_1 \cdots v_{n} \in \{0,1\}^+$ of length at least $2$ we put $\S(v_0v_1 \cdots v_{n}) = u_0u_1 \cdots u_{n-1}$, where
$$
u_{i} = v_i + v_{i+1} \mod 2 \  \ \text{ for all } \ i \in \{0,1, \ldots, n-1\}\,.
$$
\end{dfntn}

\begin{xmpl}
Let $v = 001110$.
Then $\S(v) = \S(E(v))  = 01001$. Clearly, the images of $v$ and $E(v)$ under $\S$ coincide for each $v \in \{0,1\}^+$.
Moreover, $\S(x) = \S(y)$ if and only if $x = y$ or $x = E(y)$.
\end{xmpl}

If we extend the domain of $\S$ naturally to $\{0,1\}^\N$, Proposition \ref{Prop_SturmRote} says: $\vv$ is a Rote sequence if and only if $\S(\vv)$ is a Sturmian sequence.
Moreover, for any Sturmian sequence $\uu$ there exist two Rote sequences $\vv$ and $E(\vv)$ such that  $\uu = \S(\vv)  = \S(E(\vv))$.
Since a permutation of letters in the sequence does not influence its derivated sequences, we will work only with Rote sequences starting with the letter $0$ without lose of generality.
We will also use the bar notation $\bar{\vv}=E(\vv)$  or $\bar{v} =E(v)$ to express the sequence or the word with exchanged letters $0 \leftrightarrow 1$.

\medskip

\noindent {\bf Convention. } We consider only  Rote sequences $\vv\in \{0,1\}^\mathbb{N}$ with the prefix $0$. 
If a Sturmian sequence $\uu \in \{0,1\}^\mathbb{N}$   satisfies $\uu = \S(\vv)$, we say that $\vv$  is \textit{associated} with $\uu $ or equivalently $\uu$ is \textit{associated} with $\vv$.

\medskip

To a given word $u \in \{0,1\}^*$ there are exactly two words $v, \bar{v}$ such that $S(v) = S(\bar{v}) = u$.
Moreover, if the first letter of $v$ is given, then the rest of the word $v = v_0 \cdots v_n$ is completely determined by $u = u_0\cdots u_{n-1}$: 
\begin{equation} \label{Eq_InputLetter}
v_{i+1} = v_0 + u_0 + u_1 + \cdots + u_{i} \mod 2 \quad \text{for all } i \in \{0, 1, \ldots, n-1\}\,.
\end{equation}

\begin{lmm} \label{Prop_Factors}
Let $\uu$ be a Sturmian sequence associated with a Rote sequence $\vv$.
A word $u$ is a factor of $\uu$ if and only if both words $v, \bar{v}$ such that $u = \S(v) = \S(\bar{v})$ are factors of $\vv$.
Moreover, for every $m \in \N$, the index $m$ is an occurrence of $u$ in $\uu$ if and only if $m$ is an occurrence of $v$  in $\vv$ or an occurrence of $\bar{v}$ in $\vv$.
\end{lmm}

Bispecial factors of a sequence $\uu$ play  a crucial role  in finding its derivated sequences.
We use the terminology introduced by Cassaigne \cite{Cas97} to distinguish three types of bispecial factors.
Let $w$ be a bispecial factor of $\uu$.
Then the \textit{bilateral order} of $w$ is the number
$$
B(w) = \#\{(a,b) \in \A \times \A : awb \in \mathcal{L}(\uu)\} - \#\{a \in \A : aw \in \mathcal{L}(\uu)\} - \#\{b \in \A : wb \in \mathcal{L}(\uu)\} + 1\, .
$$
The bispecial factor $w$ is \textit{weak} if $B(w) < 0$, it is \textit{ordinary} if $B(w) = 0$ and it is \textit{strong} if $B(w) > 0$.

\begin{crllr} \label{Coro_Bispecials}
Let $\uu$ be a Sturmian sequence associated with a Rote sequence $\vv$, let $\ell \in \N$.
If $w$ is a bispecial factor of length $\ell$ in $\uu$, then there are two bispecial factors $x$, $\bar{x}$ of length $\ell+1$ in $\vv$ such that $w = \S(x) = \S(\bar{x})$.
Conversely, if $x$ is a bispecial factor of length $\ell+1$ in $\vv$, then $\S(x)$ of length $\ell$ is a bispecial factor in $\uu$.
Moreover, each non-empty bispecial factor of $\vv$ is ordinary and the empty word is a strong bispecial factor of $\vv$.
\end{crllr}

\begin{proof}
Let $w$ be a bispecial factor of $\uu$.
By the well known balance properties of Sturmian sequences, the bispecial factor $w$ is ordinary. 
Indeed, the words $0w1, 1w0$ are always factors of $\uu$, in addition, just one word from $\{1w1, 0w0\}$ is a factor of $\uu$. 
Without lose of generality let us suppose that $1w1 \in \mathcal{L}(\uu)$.
The associated factors of the Rote sequence $\vv$ are $0x\bar{a}$, $1\bar{x}a$, $0\bar{x}\bar{a}$, $1xa$, $0\bar{x}a$ and $1x\bar{a}$, where $w = \S(x)$ and $x$ starts with $0$ and ends with $a$.
Combining with Lemma \ref{Prop_Factors} we get that both words $x, \bar{x}$ are ordinary bispecial factors of $\vv$.

Conversely, let us suppose that $x$ is a non-empty bispecial factor of $\vv$.
It means that the words $0x$, $1x$, $x0$, $x1$ are factors of $\vv$.
Then $\S(0x) = aw$, $\S(1x) = \bar{a}w$, $\S(x0) = wb$, $\S(x1) = w\bar{b}$, where $w = \S(x)$, $a$ is the first letter of $x$ and $b$ is the last letter of $x$.
Thus $w$ is a bispecial factor of $\uu$.

Since $00,11,01,10 \in \mathcal{L}(\vv)$, the bilateral order of $\varepsilon$ is 1, i.e. $\varepsilon$  is strong.
\end{proof}

\section{Return words to prefixes of complementary symmetric  Rote sequences} \label{S_ReturnWords}

Complementary symmetric Rote sequences form a special subclass of binary sequences coding the rotations.
The return words in the sequences coding the rotations were studied  in \cite{BlBrLaVui09} in particular for palindromic factors.
To compute the exact number of return words to a factor of a given Rote sequence, we use the following results from \cite{BaPeSt08} (Lemmas 4.2 and 4.4):

\begin{itemize}
\item[i)] If $\vv$ is uniformly recurrent sequence with no weak bispecial factor, then
$\#\mathcal{R}_\vv(x) \geq 1 + \Delta\mathcal{C}_\vv(|x|)$
for every factor $x \in \mathcal{L}(\vv)$.
\item[ii)] If $\vv$ has no weak bispecial factor and $\Delta \mathcal{C}_\vv(n) < m$ for all $n \geq 0$, then
$\#\mathcal{R}_\vv(w) \leq m$
for every factor  $w \in \mathcal{L}(\vv)$.
\end{itemize}
Recall that  $\Delta \mathcal{C}_\vv$ denotes the first difference of the factor complexity $\mathcal{C}_\vv$, i.e. $\Delta \mathcal{C}_\vv(n) = \mathcal{C}_\vv(n+1) - \mathcal{C}_\vv(n)$ for each $n \in \mathbb{N}$.

\begin{thrm} \label{Thm_ReturnWordsNumber}
Let $\vv$ be a Rote sequence. Then every non-empty prefix $x$ of $\vv$ has exactly three distinct return words.
\end{thrm}

\begin{proof}
By Corollary \ref{Coro_Bispecials}, no bispecial factor of a Rote sequence $\vv$ is weak.
Every Rote sequence is uniformly recurrent and for all $n \geq 1$ it holds true
$\Delta \mathcal{C}_\vv(n)= 2$.
Thus by Lemma 4.2 from \cite{BaPeSt08}, we have $\#\mathcal{R}_\vv(x) \geq 3$ for every non-empty prefix $x$ of $\vv$.

On the other hand, since $\Delta\mathcal{C}_\vv(n) < 3$ for all $n \geq 0$, by Lemma 4.4 from \cite{BaPeSt08} we have $\#\mathcal{R}_\vv(x) \leq 3$ for every prefix $x$ of $\vv$.
Therefore, $\#\mathcal{R}_\vv(x) = 3$ for every non-empty prefix $x$ of $\vv$.
\end{proof}

\begin{rmrk}\label{DolcePerrin}
The previous theorem also follows from a more general result obtained by Dolce and Perrin in \cite{DoPe}.
They studied the so-called neutral sets. 
By our Corollary \ref{Coro_Bispecials}, the language $\mathcal{L}$ of a Rote sequence is a neutral set with the characteristic  $\chi(\mathcal{L})=0$.
As the language $\mathcal{L}$ is uniformly recurrent, we can apply  Corollary  5.4  of  \cite{DoPe} to deduce that any non-empty factor of a Rote sequence has exactly three return words.
\end{rmrk}

A direct consequence  of Theorem \ref{Thm_ReturnWordsNumber} is that all derivated sequences of a Rote sequence to its non-empty prefixes are over a ternary alphabet.
However, to study derivated sequences we need to know also the structure of return words, not only their number.

For this purpose we now describe the crucial relation between return words of Sturmian and Rote sequences.
Suppose that $\vv$ is a Rote sequence with a prefix $x$.
Then by Proposition \ref{Prop_SturmRote} and Lemma \ref{Prop_Factors}, $\uu = \S(\vv)$ is a Sturmian sequence, $w = \S(x)$ is a prefix of $\uu$ and the occurrences of $w$ in $\uu$ coincide with the occurrences of $x$ and $\bar{x}$ in $\vv$.
Let $r$, $s$ be two return words to $w$ in $\uu$, $r$ is the most frequent one.
Our aim is to find three return words to $x$ in $\vv$. We start with an example.

\begin{xmpl} \label{Exam_BBetaB1}
Consider the Sturmian sequence $\uu = u_0u_1 \cdots$ which is fixed by the Sturmian morphism $\psi: 0 \to 010$, $1 \to 01001$, i.e.
$$\uu = 010 01001 010 010 01001 010 010 01001 010 01001 010 \cdots \, .$$
The associated Rote sequence $\vv = v_0v_1 \cdots$ (i.e. $\uu = \S(\vv)$) starting with $0$ is
$$\vv = 001110001100011100011000111000110001110011 \cdots \,.$$

Take the prefix $w = 0$ of $\uu$. 
It has two return words $r = 01$, $s = 0$ and the occurrences of $w$ in $\uu$ are $0, 2,3,5,6,8,10,11, \ldots$
The associated prefix of $\vv$ is $x = 00$ since $0 = \S(00)$. 
As we know from Lemma  \ref{Prop_Factors}, the occurrences of $w = 0$ in $\uu$ correspond to the occurrences of $x = 00$ and $\bar{x} = 11$ in $\vv$.
To find the return words to $x = 00$ we have to determine precisely when the words $00$ and $11$ occur in $\vv$.

Clearly, there is the factor $00$ at the position $0$, i.e. $v_0v_1 = 00$.
Which word from $\{00, 11\}$ starts at the position $2$ depends only on the letter $v_2$, see Equation (\ref{Eq_InputLetter}). 
This letter is completely determined by the prefix of $\uu$ of length $2$, which is $u_0u_1 = 01$ (this is also the first return word to $0$ in $\uu$).
Indeed, $v_2 = v_0 + u_0 + u_1 \mod 2$.
Since $v_2 = 0 + 0 + 1 = 1$, there is the factor $11$ starting at position $2$, i.e. $v_2v_3 = 11$. 
In other words, the return word $01$ causes the alternation of the factors $x$ and $\bar{x}$, since it has an odd number of $1$'s.

To determine the factor $v_3v_4$ starting at position $3$ we have to compute the letter $v_3 = v_0 + u_0 + u_1 + u_2 = v_2 + u_2 \mod 2$. 
Since $v_2 = 0$, we get $v_3 = 1$ and $v_3v_4 = 11$. 
Notice that the word $u_2$ is the second return word to $0$ in $\uu$. Since $u_2$ has an even number of $1$'s, it leaves the factors $x$, $\bar{x}$ unchanged.  
In the next step we get $v_5v_6 = 00$, since $v_5 = v_3 + u_3 + u_4 = 1 + 0 + 1  = 0 \mod 2$. 
So we find the first return word to $00$ in $\vv$, it is the word $v_0v_1v_2v_3v_4 = 00111$.

Similarly we get $v_6 = v_5 + u_5 = 0 + 0 = 0$ and thus $v_6v_7 = 00$, so the next return word to $00$ in $\vv$ is the word $v_5 = 0$.

As $v_8 = v_6 + u_6 +u_7 = 0 + 0 + 1 = 1$, it holds true $v_8v_9 = 11$. So we have to wait until another factor $01$ appears in $\uu$.
It happens immediately since $u_8u_9 = 01$. Thus $v_{10} = v_8 + u_8 + u_9 = 0\mod 2$ and $v_{10}v_{11} = 00$.
Therefore the word $v_6v_7v_8v_9 = 0011$ is the last return word to $00$ in $\vv$.

In total, the prefix $x = 00$ of $\vv$ has three return words $0$, $0011$ and $00111$.
\end{xmpl}

As we have seen in Example \ref{Exam_BBetaB1}, to describe the return words to $x$, we have to distinguish if a given return word to $w$ causes the alternation of the factors $x$, $\bar{x}$ or not.
This is the meaning of the following definition.

\begin{dfntn}
A word $u = u_0u_1 \cdots u_{n-1} \in \{0, 1\}^*$ is called \emph{stable} (S) if  $|u|_1 = 0 \mod 2$. Otherwise, $u$ is \emph{unstable} (U).
\end{dfntn}

\begin{xmpl} \label{Exam_Stability}
The word $u = 0110101$ is stable while the word $v = 011010$ is unstable.
\end{xmpl}

\begin{rmrk}
In the notion of Parikh vectors, the factor $u$ is stable if its Parikh vector $V_u = {p \choose 0} \mod 2$ and it is unstable if $V_u = {p \choose 1} \mod 2$ for some number $p \in \{0,1\}$.
\end{rmrk}

\begin{lmm}\label{occurrences} 
Let $\vv$ be a Rote sequence and let $x$ be its prefix. Denote $\uu = \S(\vv)$  and $w =\S(x)$.  
An index $m$ is an occurrence of $x$ in $\vv$ if and only if $m$ is an occurrence of $w$ in $\uu$ and the prefix $u=u_0u_1\cdots u_{m-1}$ of $\uu$ is stable.
\end{lmm}

\begin{proof}
Recall that $u_i =v_{i+1}  + v_i \mod 2$ holds true for  all $i\in \mathbb{N}$.
By summing up$\mod 2$  we get for the prefix $u=u_0u_1\cdots u_{m-1}$ of $\uu$:
\begin{equation}\label{inputletter}
|u|_1= \sum_{i=0}^{m-1} u_i = \sum_{i=0}^{m-1}(v_{i+1}+v_i) = v_m+v_0 \mod 2.
\end{equation}
By Lemma \ref{Prop_Factors}, $m$ is an occurrence of the prefix $x$ in $\vv$ if and only if $m$ is an occurrence of $w$ in $\uu$ and the letter $v_m$ coincides with $v_0$, which is the first letter of $x$.
The equation \eqref{inputletter}  says that the letters $v_0$ and $v_m$ coincide if and only if the prefix of $\uu$ of length $m$ is stable.
\end{proof}

We have seen that the form of return words in a Rote sequence depends on the stability of the return words in the associated Sturmian sequence.
The following definition sorts the prefixes of standard Sturmian sequences according to the stability of their return words.

\begin{dfntn}\label{rsk}
Let $w$ be a prefix of a standard Sturmian sequence $\uu$ with return words $\mathcal{R}_\uu(w) = \{r, s\}$, where $r$ is the most frequent return word. 
Let $k$ be a positive integer such that $\uu$ is a concatenation of blocks $r^ks$ and $r^{k+1}s$.
We distinguish three cases:
\begin{itemize}
\item[i)] $w$ is of type $SU(k)$, if $r$ is stable and $s$ is unstable;
\item[ii)] $w$ is of type $US(k)$, if $r$ is unstable and $s$ is stable;
\item[iii)] $w$ is of type $UU(k)$, if both $r$ and $s$ are unstable. \end{itemize}
The type of the prefix $w$ is denoted $\mathcal{T}_w$ (or $\mathcal{T}$ if the respective factor $w$ is clear). 
If the number $k$ is not essential, we write only $SU$, $US$ and $UU$.
\end{dfntn}

\begin{rmrk}
It is easy to verify that all these types appear in the case of prefixes of Sturmian sequences.
On the other hand, the fourth possible case, i.e. the type $SS$, cannot appear. 
We can prove this using the results from \cite{BaBuLuHlPu16}. It also follows from the proof of Theorem 4 in \cite{Ro94}. 

First we recall the WELLDOC property. 
A sequence $\uu \in \{0,1\}^\N$ has \textit{well distributed occurrences modulo 2} (shortly WELLDOC(2) property) if for every factor $w \in \mathcal{L}(\uu)$ we have
$$
\left\{
\left(\begin{array}{c}
|u|_0 \\
|u|_1 \\
\end{array} \right)\! \! \! \! \! \mod 2
  : uw \ \text{is a prefix of}\ \uu \right\} =
\left\{
\left(\begin{array}{c}
0 \\
0 \\
\end{array} \right) \, ,
\left(\begin{array}{c}
0 \\
1 \\
\end{array} \right) \, ,
\left(\begin{array}{c}
1 \\
0 \\
\end{array} \right) \, ,
\left(\begin{array}{c}
1 \\
1 \\
\end{array} \right) \,
\right\} \,.
$$
As shown in \cite{BaBuLuHlPu16}, all Sturmian sequences have the WELLDOC(2) property.

Let us suppose that $w$ is a prefix of $\uu$ with two stable return words, i.e. the numbers of $1$'s occurring in $r$ and $s$ are even.
Since any word $u$ such that $uw$ is a prefix of $\uu$ is a concatenation of words $r$ and $s$, $u$ contains an even number of $1$'s. It contradicts the WELLDOC(2) property of $\uu$.
\end{rmrk}

We use these prefix types to describe the return words to corresponding Rote prefixes.

\begin{thrm} \label{Thm_ReturnWords}
Let $\vv$ be a Rote sequence associated with a standard Sturmian sequence $\uu = \S(\vv)$. 
Let $x$ be a non-empty prefix of $\vv$ and $w=\S(x)$.
Then the prefix $x$ of $\vv$ has three return words $A, B, C \in \{0,1\}^+$ satisfying ($r$, $s$ and $k$  are the same as in Definition \ref{rsk}):
\begin{itemize}
\item[i)]  if $\mathcal{T}_w=  SU(k)$,     then  \ \ $\S(A0)=r$, \  \ \ $ \S(B0)=sr^{k+1}s $ \ \ and  \ \  $\S(C0)=sr^ks $;
\item[ii)]  if $\mathcal{T}_w=  US(k)$,     then  \ \ $\S(A0)=rr$, \  \ $ \S(B0)=rsr $ \ \ \ \ \ \,and  \ \  $\S(C0)=s $;
\item[iii)]  if $\mathcal{T}_w=  UU(k)$,     then  \ \ $\S(A0)=rr$, \  \ $ \S(B0)=rs $ \ \ \ \ \ \ \,and  \ \  $\S(C0)=sr $.
\end{itemize}
\end{thrm}

\begin{proof}
Let us suppose that $\mathcal{T}_w = SU(k)$, i.e. $|r|_1 = 0 \mod 2$ and $|s|_1 = 1 \mod 2$.
Let $n $ be an occurrence of $x$ in $\vv$.
Then by Lemma \ref{occurrences} the index $n$ is an occurrence of $w$ in $\uu$ and the prefix $u=u_0u_1\ldots u_{n-1}$ is stable. 
Since $\uu$ is a concatenation of the blocks $r^{k+1}s$ and $r^{k}s$, the sequence  $\uu$ has one of the prefixes $ur$, $usr^{k+1}s$ or $usr^{k}s$.

-- \ \ If $ur$ is a prefix of $\uu$, then $n+|r|$ is an occurrence of $w$ in $\uu$.
Moreover, the prefix  of $\uu$ of length $n+|r|$ is stable.
It means that $m: =n+|r|$ is the subsequent occurrence of $x$ in $\vv$ and $A : = v_n v_{n+1}\cdots v_{m-1}$ is a return word to $x$ in $\vv$.
Let us recall our convention that $0$ is a prefix of $\vv$ and thus any return word to the prefix $x$ begins with $0$, in particular $v_m=0$.
Therefore, $r= u_nu_{n+1}\cdots u_{m-1} = \S(A0)$.

\medskip

--  \ \ If $usr^{k+1}s$  is a prefix of $\uu$, then  any index $\ell \in \{n+|s|, n+|s| +|r|, n+|s|+2|r|, \ldots, n+|s|+(k+1)|r|\} $ is an occurrence of $w$ in $\uu$.
Since $r$ is stable and $s$ is unstable, prefixes of these lengths $\ell$ are unstable and by Lemma \ref{occurrences}, such a index $\ell$ is not an occurrence of $x$ in $\vv$.
The next occurrence of $w$ in $\uu$ is  $m:=n+|s|+(k+1)|r|+|s|$.
The prefix of $\uu$ of length $m$ is stable and thus $m$ is the smallest occurrence of $x$ in $\vv$ grater than $n$.
Therefore $B:=v_n \ldots v_{m-1}$ is a return word to $x$ in $\vv$ and obviously $sr^{k+1}s  = \S(B0)$.

\medskip
The reasoning  in all remaining cases is  analogous and so we omit it.
\end{proof}

\begin{xmpl}[Example \ref{Exam_BBetaB1} continued]
Recall that the prefix $00$ of $\vv$ has three return words $A = 0$, $B = 0011$ and $C = 00111$.
The associated Sturmian prefix $S(00) = 0$ has the return words $r = 01$, $s = 0$ and $\uu$ is a concatenation of blocks $rs = 010$ and $rrs = 01010$.
Thus the type of $0$ is $\mathcal{T}_0 = US(1)$.
It holds true $\S(A0) = \S(00) = 0 = s$, $\S(B0) = \S(00110) = 0101 = rr$  and $\S(C0) = \S(001110) = 01001 = rsr$.
\end{xmpl}

It remains to explain how to determine the type of a given prefix $w$ of a standard Sturmian sequence $\uu$.
This question will be solved in Chapter \ref{types}.

\section{Derivated sequences of complementary symmetric Rote sequences} \label{Sec_DerivatedSequences}

As we have proved in Theorem \ref{Thm_ReturnWordsNumber}, any derivated sequence of a Rote sequence $\vv$ is over a ternary alphabet (we use the alphabet $\{A, B, C\}$).
In this section we study the structure of these ternary sequences in the case that $\vv$ is associated with a standard Sturmian sequence.
First we mention an important direct consequence of Theorem \ref{Thm_ReturnWords}.

\begin{crllr} \label{Thm_DetermineDerivatedWord}
Let $\vv$ be a Rote sequence  associated with  a standard Sturmian sequence  $\uu = \S(\vv)$ and let $x$ be a non-empty prefix of $\vv$.
Then the derivated sequence $\dd_\vv(x)$ is uniquely determined by the derivated sequence $\dd_\uu(w)$ of $\uu$ to the prefix $w = \S(x)$ and by the type $\mathcal{T}_w$ of the prefix $w$.
\end{crllr}

\begin{proof}
Let $r, s$ be the return words to $w$ in $\uu$ and let $\uu$ be a concatenation of blocks $r^ks$ and $r^{k+1}s$ for some positive integer $k$.
We decompose the sequence $\dd_\uu(w) \in \{r, s\}^\N$ from the left to the right into three types of blocks $\S(A0),\S( B0)$ and $\S(C0)$ according to the type $\mathcal{T}_w$ (the relevant blocks are listed in Theorem \ref{Thm_ReturnWords}).
Then the order of letters $A,B,C$ in this decomposition is the desired derivated sequence $\dd_{\vv}(x)$ of $\vv$ to $x$.
It remains to explain that this decomposition is unique.
In the case i) we decompose $\dd_{\uu}(w)$ into the minimal blocks with an even number of letter $s$, similarly in the case ii) we decompose $\dd_{\uu}(w)$ into the minimal blocks with an even number of letter $r$.
In the case iii) we decompose $\dd_{\uu}(w)$ into the pairs of letters.
\end{proof}

The main goal of this section is to show that any derivated sequence $\dd_\vv(x)$ of $\vv$ is in fact coding of a three interval exchange transformation.  
The sequences coding the interval exchange transformation were introduced in \cite{Ose66} and they are intensively studied as they represent an important generalization of Sturmian  sequences to the multi-literal alphabets, see  \cite{Rau79}. 
Here we define only those interval exchange  transformations which appear in our description of derivated sequences $\dd_\vv(x)$.

A three interval exchange transformation $T: \left[0,1\right) \to \left[0,1\right)$ is given by two parameters $\beta, \gamma  \in (0,1)$, $\beta+\gamma <1$,  and by a permutation $\pi$ on the set $\{1, 2, 3\}$.
The interval $\left[ 0,1\right) $ is partitioned into three subintervals
$$
I_A = \left[0, \beta\right), \quad   I_B = \left[\beta, \beta + \gamma\right) \quad \text{and} \quad  I_C=\left[\beta+ \gamma, 1\right)
$$
of lengths $\beta$, $\gamma$ and $1 - \beta - \gamma$ respectively.
These intervals are then rearranged by the transformation $T$ according to the permutation $\pi$. 
More specifically:
\begin{itemize}
\item If the permutation $\pi=(3,2,1)$, then
$$
T(y) =
\left\{
\begin{array}{ll}
y + 1 - \beta \ \ \ \ \ \ \ \ \ \  & \text{if}\  y \in I_A\,, \\
y + 1 - 2\beta- \gamma & \text{if} \ y \in I_B\,, \\
y - \beta - \gamma&\text{if}\  y \in I_C\,. \\
\end{array}
\right.
$$
\item If the permutation $\pi=(2,3,1)$, then
$$
T(y) =
\left\{
\begin{array}{ll}
y + 1 - \beta \ \ \ \ \ \ \ \ \ \  &\text{if}\  y \in I_A\,, \\
y -\beta  & \text{if}\  y \in I_B\,, \\
y - \beta & \text{if}\  y \in I_C\,. \\\end{array}
\right.
$$
\end{itemize}

Let $ \rho \in \left[ 0,1 \right)$. 
The sequence $\uu = u_0u_1u_2 \cdots \in \{A,B,C\}^\N$ defined by
$$
u_n =
\left\{
\begin{array}{ll}
 A \ \ \ \ \ \ \ & \text{if} \ T^n(\rho) \in I_A\,,\\
 B & \text{if} \ T^n(\rho) \in I_B\,,\\
C & \text{if} \ T^n(\rho)  \in I_C\,
\end{array}\right.
$$
is called a \textit{3iet sequence} coding the intercept $\rho$ under the transformation $T$.

Take a standard Sturmian sequence $\uu$.
As we have mentioned before, every derivated sequence $\dd_\uu(w)$ of $\uu$ to a given prefix $w$ is also a standard Sturmian sequence.
Thus $\dd_\uu(w)$ is expressible as a 2iet sequence with the slope $\alpha$ and the intercept $\rho = 1- \alpha$.

\begin{prpstn}\label{Prop_ThreeExchange}
Let $\vv$ be a Rote sequence  associated with a standard Sturmian sequence  $\uu = \S(\vv)$, let $x$ be a non-empty prefix of $\vv$ and $w=\S(x)$.
Let $\alpha> \tfrac12$  be the slope of the Sturmian sequence $\dd_\uu(w)$.
Then the derivated sequence $\dd_\vv(x)$ is a 3iet sequence coding the intercept $\rho = 1-\alpha$ under the three interval exchange transformation $T$, where $T$ is given by the  following  parameters $\beta, \gamma$ and permutation $\pi$:
\begin{itemize}
\item[i)]  if $\mathcal{T}_w  = SU(k)$, then $\beta=\alpha$, $\gamma= \alpha -k(1- \alpha)$,   and  $\pi = (3,2,1)$;
\item[ii)]  if $\mathcal{T}_w  = US(k)$,  then $\beta=2\alpha-1$, $\gamma= 1-\alpha$,  and  $\pi = (3,2,1)$;
\item[iii)] if $\mathcal{T}_w  = UU(k)$,  then $\beta=2\alpha-1$, $\gamma= 1-\alpha$,  and   $\pi = (2,3,1)$.
\end{itemize}
\end{prpstn}

\begin{proof}
Since any derivated sequence of a standard Sturmian sequence is standard as well, $\dd_\uu(w)$  is coding of the intercept $1-\alpha$ under the transformation $G: [0,1) \to [0,1)$ defined by
$$
G(y)= y+1-\alpha, \ \ \text {if} \ \ y\in I_r=[0, \alpha)\quad \text{and} \quad G(y)= y-\alpha, \ \ \text {if} \ \ y\in I_s=[\alpha,1)\,.
$$

Let us start with the simplest case iii): 
\ By Theorem \ref{Thm_ReturnWords}, the derivated sequence $\dd_{\vv}(x)$ of $\vv$ to the prefix $x$ is determined by the decomposition of $\dd_{\uu}(w)$ into blocks of length $2$.
The order of blocks $rr$, $rs$ and $sr$ in the decomposition of $\dd_{\uu}(w)$ is given by the transformation $G^2$ under which the point $\rho =1-\alpha$ is coded.
A simple computation gives:
$$
G^2(y) =
\left\{
\begin{array}{lll}
y + 2 - 2\alpha \ \ & \text{if}\  &y \in  [0, 2\alpha-1)\,, \\
y +1-2\alpha  & \text{if}\  &y \in [2\alpha-1, \alpha)\,, \\
y +1-2\alpha  & \text{if}\  &y \in [\alpha, 1)\,. \\
\end{array}
\right.
$$
It means that $G^2$ exchanges three intervals under the permutation $(2,3,1)$  with the parameters $\beta, \gamma$ as claimed in point iii) of the statement.

\medskip

Let $w$ be of type $SU(k)$ as assumed in i). 
Let us denote the intervals 
$$
I_A= [0, \alpha), \ \  I_B=[\alpha, 2\alpha -k(1-\alpha)), \ \ I_C = [ 2\alpha -k(1-\alpha),1)
$$
and define the transformation
$$
T(y) =
\left\{
\begin{array}{lll}
G(y)  \ \ \ \ \ \ \ \ &  \text{if}\ & y \in  I_A\,, \\
G^{k+3}(y)&  \text{if}\  & y \in I_B\,, \\
G^{k+2}(y)& \text{if}\  & y \in I_C\,. \\
\end{array}
\right.
$$

Recall that the parameter $k$ in the type of $w$ means that  $\dd_\uu(w)$ is a concatenation of blocks $r^ks$ and $r^{k+1}s$.
By Theorem \ref{Thm_ReturnWords}, the derivated sequence $\dd_{\vv}(x)$ of $\vv$ to the prefix $x$ is determined by the unique decomposition of $\dd_{\uu}(w)$ into blocks $r$, $sr^{k+1}s$ and $sr^ks$.
As mentioned in Section \ref{S_SturmianSequences}, $k= \lfloor \tfrac{\alpha}{1-\alpha}\rfloor$, i.e.  $\alpha > k(1-\alpha)$ and $\alpha < (k+1)(1-\alpha)$.
Therefore the intervals $I_A,I_B$, and $I_C$ are well defined.

\medskip

\noindent To prove i), one has to check

 \begin{enumerate}
\item \quad $I_A  \subset I_r$;
\item  \quad  $I_B  \subset I_s, \   \ G^j(I_B) \subset I_r \ \text{for all } j =1,2,\ldots, k+1, \   \ G^{k+2}(I_B) \subset I_s$;
\item \quad  $I_C  \subset I_s, \   \ G^j(I_C) \subset I_r \ \text{for all } j =1,2,\ldots, k, \   \ G^{k+1}(I_C) \subset I_s$;
\item \quad $T$ is an interval exchange transformation under the permutation (3,2,1), i.e., \\[1mm]
{$T(I_A) = \bigl[1-\alpha, 1\bigr)$, \ \  $T(I_B) = \bigl[ (k+1)(1-\alpha) -\alpha, 1-\alpha\bigr)$, \ \  $ T(I_C) = \bigl[0,  (k+1)(1-\alpha) -\alpha\bigr)$ \ . }
\end{enumerate}
Validity of  (1)--(4) follows directly from the definition of $G$.

\medskip

Proof of Point ii) is analogous.
\end{proof}

\begin{rmrk}
It can be shown that all three transformations $T$ from Proposition \ref{Prop_ThreeExchange} satisfy the so called   \textit{i.d.o.c. property} \cite{Kea75}.
For a three interval exchange transformation with the discontinuity points $\beta$ and $\beta+\gamma$ it means that $T^n(\beta)\neq \beta+\gamma$ for all $n \in \mathbb{Z}$.
Property  i.d.o.c. implies that the factor complexity of any  derivated sequence $\dd_\vv(x)$  is $\mathcal{C}(n) =2n+1$.
\end{rmrk}

\begin{crllr}\label{theSame}
Let $\vv$ be a Rote sequence  associated with  a standard Sturmian sequence  $\uu = \S(\vv)$ and let $x, x'$ be two non-empty prefixes of $\vv$.
Denote $w=\S(x)$  and $w'=\S(x')$. The derivated sequence of $\vv$ to the prefix $x$ coincides with the derivated sequence of $\vv$ to the prefix $x'$ if and only if the types of $w$ and $w'$ are the same and the derivated sequence of $\uu$ to the prefix $w$ coincides with the derivated sequence of $\uu$ to the prefix $w'$.
In other words,
$$\dd_{\vv}(x) = \dd_{\vv}(x') \quad \quad \quad \text{ iff} \quad \quad\quad  \mathcal{T}_{w} = \mathcal{T}_{w'} \quad \text{and}\quad \dd_{\uu}(w) = \dd_{\uu}(w') \,.$$
\end{crllr}

\begin{proof}
Let us assume that  $\dd_\vv(x)=\dd_\vv(x')$.
We use two well known properties of 3iet sequences, see for example \cite{FeHaZa03} and \cite{FeMo10}:

--  the frequencies of letters in a 3iet sequence correspond to the lengths of the intervals $I_A$, $I_B$  and  $I_C$;

-- the  language of a 3iet sequence is closed under reversal if and only if the permutation is $(3,2,1)$.

\medskip

By Proposition \ref{Prop_ThreeExchange}, the language of $\dd_\vv(x)$ is not closed under reversal if and only if $w = \S(x)$ is of type $UU$.
Moreover, if $w$ is of type $SU$, the frequencies of letters are: $\alpha$, $\alpha - k(1-\alpha)$ and  $(k+1)(1-\alpha) - \alpha$.
Since $\alpha$ is irrational, these three lengths are pairwise distinct. 
If $w$ is of type $US$ or $UU$, the letters $B$ and $C$ have the same frequency $1-\alpha$.
Therefore, the assumption $\dd_\vv(x)=\dd_\vv(x')$ implies that the type of $w$ and the type of $w'$  are the same.
Moreover, the lengths of the intervals $I_A, I_B, I_C$, i.e. the frequencies of the letters, must be the same. 
It implies that the slopes of $\dd_\uu(w)$  and $\dd_\uu(w')$ are equal.
Since $\dd_\uu(w)$ and $\dd_\uu(w')$ are both standard Sturmian sequences with the same slope, obviously $\dd_\uu(w)=\dd_\uu(w')$.

The opposite implication follows from Corollary \ref{Thm_DetermineDerivatedWord}.
\end{proof}

The proof of Corollary \ref{Thm_DetermineDerivatedWord} gives us the instructions how to construct the derivated sequences of a Rote sequence: we need to know both the derivated sequences of the associated Sturmian sequence $\uu$ and the types of prefixes of $\uu$.
Remind that we work only with Rote sequences associated with standard Sturmian sequences, thus $\uu$ is always standard and any prefix of $\uu$ is left special.
Due to \eqref{onlyRight} and  Corollary \ref{Coro_Bispecials}, we can focus only on the bispecial prefixes of standard Sturmian sequences.

\section{Types of bispecial prefixes of Sturmian sequences}\label{types}

Consider a standard Sturmian sequence $\uu$ with the directive sequence $\zz \in \{b, \beta\}^\N$.
It means that there is a sequence $(\uu^{(n)})_{n \geq 0}$ of standard Sturmian sequences such that for every $n \in \N$
\begin{equation}\label{preimage}
\uu = \varphi_{z_0z_1 \ldots z_{n-1}}(\uu^{(n)})\,.
\end{equation}
\noindent {\bf Convention. } We order the bispecial prefixes of $\uu$ by their length and we denote the $n^{th}$ bispecial prefix of $\uu$ by $w^{(n)}$. 
In particular, $w^{(0)} =
\varepsilon$,  $w^{(1)} = 0$ if $z_0 = b$ and $w^{(1)} = 1$ if $z_0 = \beta$.
\medskip

Our aim is to find for each $n \in \mathbb{N}$ the derivated sequence of $\uu$ to the prefix $w^{(n)}$ and to determine the type of $w^{(n)}$.
First we need to know how bispecial factors and their return words change under the application of morphisms $\varphi_b$ and $\varphi_\beta$.
It is shown in \cite{KlMePeSt18}.

\begin{lmm} \label{Lem_ImageB}
Let $\uu', \uu$ be Sturmian sequences such that $\uu = \varphi_b(\uu')$.
\begin{itemize}
\item[i)] For every bispecial factor $w'$ of $\uu'$, the factor $w = \varphi_b(w')0$ is a bispecial factor of $\uu$.
\item[ii)] Every bispecial factor $w$ of $\uu$ which is not empty can be written as $w = \varphi_b(w')0$ for a uniquely given bispecial factor $w'$ of $\uu'$. 
\item[iii)] The words $r', s'$ are return words to a bispecial prefix $w'$ of $\uu'$ if and only if $r = \varphi_b(r'), s = \varphi_b(s')$ are return words to a bispecial prefix $w = \varphi_b(w')0$ of $\uu$.
Moreover,  $\dd_\uu(w) = \dd_{\uu'}(w')$.
\end{itemize}
\end{lmm}

\begin{lmm} \label{Lem_ImageBeta}
Let $\uu', \uu$ be Sturmian sequences such that $\uu = \varphi_\beta(\uu')$.
\begin{itemize}
\item[i)] For every bispecial factor $w'$ of $\uu'$, the factor $w = \varphi_\beta(w')1$ is a bispecial factor of $\uu$.
\item[ii)] Every bispecial factor $w$ of $\uu$ which is not empty can be written as $w = \varphi_\beta(w')1$ for a uniquely given bispecial factor $w'$ of $\uu'$. 
\item[iii)] The words $r', s'$ are return words to a bispecial prefix $w'$ of $\uu'$ if and only if $r = \varphi_\beta(r'), s = \varphi_\beta(s')$ are return words to a bispecial prefix $w = \varphi_\beta(w')1$ of $\uu$.
Moreover,  $\dd_\uu(w) = \dd_{\uu'}(w')$.
\end{itemize}
\end{lmm}

\begin{xmpl} \label{Exam_BBetaB3}
The sequence from Example \ref{Exam_BBetaB1} is the fixed point of the morphism $\varphi_{b \beta b}$.
So it is a standard Sturmian sequence with the directive sequence $\zz = b \beta b b \beta b b \beta b \cdots$.

The $0^{th}$ bispecial prefix of $\uu$ is the empty word $w^{(0)} = \varepsilon$.
Its return words are $0$ and $1$ and clearly $\dd_{\uu}(w^{(0)}) = \uu$.

By Lemma \ref{Lem_ImageB}, the bispecial prefix $w^{(1)}$ can be obtained from $\varepsilon$ using the morphism $\varphi_b$: $w^{(1)} = \varphi_b(\varepsilon)0 = 0$.
It means that $w^{(1)} = 0$ originates in the sequence $\uu^{(1)}$ which has the directive sequence $\beta b b \beta b b \cdots$.
The return words to $w^{(1)}$ are $\varphi_b(0) = 0$ and $\varphi_b(1) = 01$ and its derivated sequence is $\dd_{\uu}(w^{(1)}) = \dd_{\uu^{(1)}}(\varepsilon) = \uu^{(1)}$.

Similarly, the prefix $w^{(2)}$ arises from $\varepsilon$ by application of $\varphi_b \varphi_\beta$, i.e.
$$
w^{(2)} = \varphi_b (\varphi_\beta(\varepsilon)1)0 = 010\,.
$$
Thus $w^{(2)}$ originates in $\uu^{(2)}$ with the  directive sequence $b b \beta b b \beta \cdots$.
The return words to $w^{(2)}$ are $\varphi_b\varphi_\beta(0) = 010$ and $\varphi_b\varphi_\beta(1) = 01$ and its derivated sequence is $\dd_{\uu}(w^{(2)}) =  \dd_{\uu^{(2)}}(\varepsilon) = \uu^{(2)}$.
\end{xmpl}

The method explained in Example \ref{Exam_BBetaB3} can be easily generalized.
Let us formalize this procedure.

\begin{crllr}\label{composition}
Let $\zz \in \{b, \beta\}^\N$ be a directive sequence of a standard Sturmian sequence $\uu$ and let $n \in \mathbb{N}$.
Denote by $r^{(n)}$ the most frequent and by $s^{(n)}$ the less frequent return words to the $n^{th}$ bispecial prefix $w^{(n)}$ of $\uu$.
Then the derivated sequence $\dd_\uu(w^{(n)}) $ is a standard Sturmian sequence with the directive sequence $\zz^{(n)} = z_nz_{n+1}z_{n+2}\cdots$.
Moreover
\begin{itemize}
\item[i)] If $z_n = b $,  then
$r^{(n)} = \varphi_{z_0z_1 \cdots z_{n-1}}(0)$  \  and  \    $s^{(n)} = \varphi_{z_0z_1 \ldots z_{n-1}}(1)$.
\item[ii)] If $z_n = \beta $,  then
$r^{(n)} = \varphi_{z_0z_1 \cdots z_{n-1}}(1)$  \  and  \    $s^{(n)} = \varphi_{z_0z_1 \ldots z_{n-1}}(0)$.
\end{itemize}
\end{crllr}

\begin{proof}
We proceed with induction on $n\in \mathbb{N}$.

Clearly, the return words to the bispecial prefix $w^{(0)} = \varepsilon$  are letters $0$ and $1$ and thus the derivated sequence  of $\uu$ to $w^{(0)}$ is the sequence $\uu$  itself.
Using the notation of \eqref{preimage}, we have $\uu = \varphi_{z_0}(\uu^{(1)})$.
If $z_0 = b$ then $0$ is the most frequent letter in $\uu$, if $z_0 = \beta$ then $1$ is the most frequent letter of $\uu$.

The directive sequence of $\uu^{(1)}$ is $\zz^{(1)} = z_1z_{2}z_{3}\cdots$.
Denote $\uu'= \uu^{(1)}$.
Let $w'$ be the $n^{th}$ bispecial prefix of $\uu'$ and $r', s'$ be its return words.
By Lemmas \ref{Lem_ImageB} and \ref{Lem_ImageBeta}, the words $\varphi_{z_0}(r')$ and $\varphi_{z_0}(s')$ are the return words to the $(n+1)^{st}$ bispecial prefix of $\uu$ (we add $1$ for the bispecial prefix $\varepsilon$ of $\uu$).

If we now apply the induction hypothesis on $\uu'$ and take into consideration that the application of $\varphi_{z_0}$ to the return words does not change their frequencies, the statement is proved.
\end{proof}

It remains to determine the types of bispecial prefixes of standard Sturmian sequences.
We will use the following matrix formalism.

\begin{dfntn} \label{Rem_BispecialMatrix}
Let $w$ be a prefix of a standard Sturmian sequence $\uu$. 
Let $r, s$ be the return words to $w$ in $\uu$, where $r$ is the most frequent return word.
Then the matrix $P_w$ is defined as:
$$P_w =
\left(\begin{array}{cc}
|r|_0 & |s|_0\\
|r|_1 & |s|_1\\
\end{array} \right)
\mod 2\,. $$
\end{dfntn}

\begin{rmrk}
The type $\mathcal{T}_w$ of the prefix $w$  depends  on the bottom row of the matrix $P_w$. 
$\mathcal{T}_w$ is
\begin{itemize}
\item[i)] $SU \ $ if
$P_w = \left(\begin{array}{cc}
p & q\\
0 & 1\\
\end{array} \right) \ \ $ for some numbers $p, q \in \{0,1\}$;
\item[ii)]$US\ $ if
$P_w = \left(\begin{array}{cc}
p & q\\
1 & 0\\
\end{array} \right) \ \ $ for some numbers $p, q \in \{0,1\}$;
\item[iii)] $UU\ $ if
$P_w = \left(\begin{array}{cc}
p & q\\
1 & 1\\
\end{array} \right) \ \ $ for some numbers $p, q \in \{0,1\}$.
\end{itemize}
\end{rmrk}

\begin{xmpl} \label{Exam_BBetaB2}
Take the prefixes $\epsilon, 0$ and $010$ of Sturmian sequence $\uu$ from Example \ref{Exam_BBetaB1}.
Their matrices are
$$
P_\varepsilon = \left(\begin{array}{cc}
1 & 0\\
0 & 1\\
\end{array} \right)
\,, \quad 
P_0 = \left(\begin{array}{cc}
1 & 1\\
1 & 0\\
\end{array} \right)
\quad \text{and} \quad 
P_{010} = \left(\begin{array}{cc}
0 & 1\\
1 & 1\\
\end{array} \right)\,
$$
and their types are $\mathcal{T}_\varepsilon = SU(1)$, $\mathcal{T}_0 = US(1)$ and  $\mathcal{T}_{010} = UU(2)$ respectively.
\end{xmpl}

\noindent {\bf Convention. }To simplify the notation, for the $n^{th}$ bispecial prefix $w^{(n)}$ of a standard Sturmian sequence we will denote its type $\mathcal{T}^{(n)}$  instead of $\mathcal{T}_{w^{(n)}}$ and its matrix $P^{(n)}$ instead of $P_{w^{(n)}}$.

\begin{bsrvtn} \label{Obs_EpsilonType}
Let $\uu$ be a standard Sturmian sequence with the directive sequence $\zz \in \{b, \beta\}^\N$.
\begin{itemize}
\item[i)] If $\uu$ has a prefix $b^k\beta$ for some positive integer $k$, then by Observation \ref{Obs_Epsilon}, $\uu$ is a concatenation of the blocks $0^k1$ and $0^{k+1}1$.
Therefore, the bispecial prefix $w^{(0)} = \varepsilon$ has the stable return word $r^{(0)} = 0$ and the  unstable return word $s^{(0)} = 1$, $w^{(0)}$ is of type $\mathcal{T}^{(0)} = SU(k)$ and its matrix $P^{(0)}$ is
$$
\left(\begin{array}{cc}
1 & 0\\
0 & 1 \\
\end{array} \right) =: O_b\,.
$$
\item[ii)] If $\uu$ has a prefix $\beta^kb$ for some positive integer $k$, then $w^{(0)} = \varepsilon$ has the unstable return word $r^{(0)} = 1$ and the stable return word $s^{(0)} = 0$, its type is $\mathcal{T}^{(0)} = US(k)$ and its matrix $P^{(0)}$ is
$$
\left(\begin{array}{cc}
0 & 1\\
1 & 0\\
\end{array} \right) =: O_\beta\,.
$$
\end{itemize}
\end{bsrvtn}

Let us recall that the Parikh vector $V_{\psi(u)}$ can be computed by multiplication $V_{\psi(u)} = M_\psi V_u$ for all $u \in \A^*$.
By Corollary \ref{composition}, the Parikh vectors of the return words $r^{(n)}$  and $s^{(n)}$ can be computed using the matrix of the morphism $\varphi_{z_0z_1 \ldots z_{n-1}}$.
Thus the matrix $P^{(n)}$ can be obtained as a product of the matrix of the morphism $\varphi_{z_0z_1 \ldots z_{n-1}}$ and the matrix $O_b$ or $O_\beta$.
The following proposition summarizes these observations. 

\begin{prpstn} \label{Prop_TypeDetermination}
Let $\uu$ be a standard Sturmian sequence with the directive sequence $\zz \in \{b, \beta\}^\N$ and let  $n \in \N$.
\begin{itemize}
\item[i)] If the sequence $z_{n}z_{n+1}z_{n+2}\cdots$ has a prefix $b^k\beta$, then
$$
P^{(n)} = M_{z_0}M_{z_1} \cdots M_{z_{n-1}}O_b \mod 2\,.
$$
\item[ii)]  If the sequence $z_{n}z_{n+1}z_{n+2}\cdots$ has a prefix $\beta^k b$, then
$$
P^{(n)} = M_{z_0}M_{z_1} \cdots M_{z_{n-1}}O_\beta \mod 2\,.
$$
\end{itemize}
The type $\mathcal{T}^{(n)}$ is given by the bottom row  of matrix $P^{(n)}$ and the number $k$.
\end{prpstn}

\begin{xmpl}[Example \ref{Exam_BBetaB3} continued] \label{Exam_BBetaB4}
The $0^{th}$ bispecial prefix of $\uu$ is $w^{(0)} = \varepsilon$.
Since $\zz$ has the prefix $b\beta$, its matrix is $O_b$ and its type is $SU(1)$.

For the bispecial prefix $w^{(1)} = 0$, we have $z_0 = b$ and $z_1z_2z_3 \cdots = \beta b b\beta \cdots$ has the prefix $\beta b$. Thus the corresponding matrix is 
$$
P^{(1)} = M_bO_\beta =
\left(\begin{array}{cc}
1 & 1\\
0 & 1\\
\end{array} \right)
\left(\begin{array}{cc}
0 & 1\\
1 & 0\\
\end{array} \right)
= \left(\begin{array}{cc}
1 & 1\\
1 & 0\\
\end{array} \right) \mod 2\, 
$$
and the $1^{st}$  bispecial prefix $w^{(1)}$ is of type $\mathcal{T}^{(1)} = US(1)$.

For the bispecial prefix $w^{(2)}$ we have $z_0z_1 = b \beta$ and $z_2z_3z_4 \cdots = b b \beta b \cdots$ has the prefix $b^2\beta$.
Therefore its matrix is

$$ P^{(2)} = M_b M_\beta O_b
=
\left(\begin{array}{cc}
1 & 1\\
0 & 1\\
\end{array} \right)
\left(\begin{array}{cc}
1 & 0\\
1 & 1\\
\end{array} \right)
\left(\begin{array}{cc}
1 & 0\\
0 & 1\\
\end{array} \right)
=
\left(\begin{array}{cc}
0 & 1\\
1 & 1\\
\end{array} \right) \mod 2\,
$$
and its type is $\mathcal{T}^{(2)} = UU(2)$.
\end{xmpl}

Finally, we study what kind of matrices can appear among  the matrices $P^{(n)}$ of Sturmian bispecial prefixes.
Clearly, all matrices $M_b, M_\beta, O_b, O_\beta$ have their determinants equal to $1$.
By Proposition \ref{Prop_TypeDetermination}, the matrix $P^{(n)}$ is a  product of these  matrices modulo $2$.
So the determinant of $P^{(n)}$ has to be equal to $1$.
Therefore there are only six candidates for $P^{(n)}$:
\begin{align} \label{Set_matrices}
\left\{
\left(\begin{array}{cc}
1 & 0\\
0 & 1\\
\end{array} \right)\, ,
\left(\begin{array}{cc}
1 & 1\\
0 & 1\\
\end{array} \right)\, , \left(\begin{array}{cc}
1 & 0\\
1 & 1\\
\end{array} \right)\,, \left(\begin{array}{cc}
1 & 1\\
1 & 0\\
\end{array} \right)\,,\left(\begin{array}{cc}
0 & 1\\
1 & 1\\
\end{array} \right)\,,  \left(\begin{array}{cc}
0 & 1\\
1 & 0\\
\end{array} \right)
\right\}\,.
\end{align}

The relations between these matrices are captured in Figure \ref{Fig_Graph}.
We can go through this graph instead of calculating the respective products mod $2$.

\begin{xmpl}
The result of the product
$P^{(2)} = M_b M_\beta O_b$
can be obtained as follows. We start in the vertex
$O_b = \left(\begin{array}{cc}
1 & 0\\
0 & 1\\
\end{array} \right)$.
Then we move along the edge labelled by $\varphi_\beta$ to the vertex
$\left(\begin{array}{cc}
1 & 0\\
1 & 1\\
\end{array} \right)$.
After that we move along the edge labelled by $\varphi_b$ to the vertex
$\left(\begin{array}{cc}
0 & 1\\
1 & 1\\
\end{array} \right)$.
This vertex is the desired matrix~$P^{(2)}$.
\end{xmpl}

\begin{figure}

\includegraphics[scale=0.35]{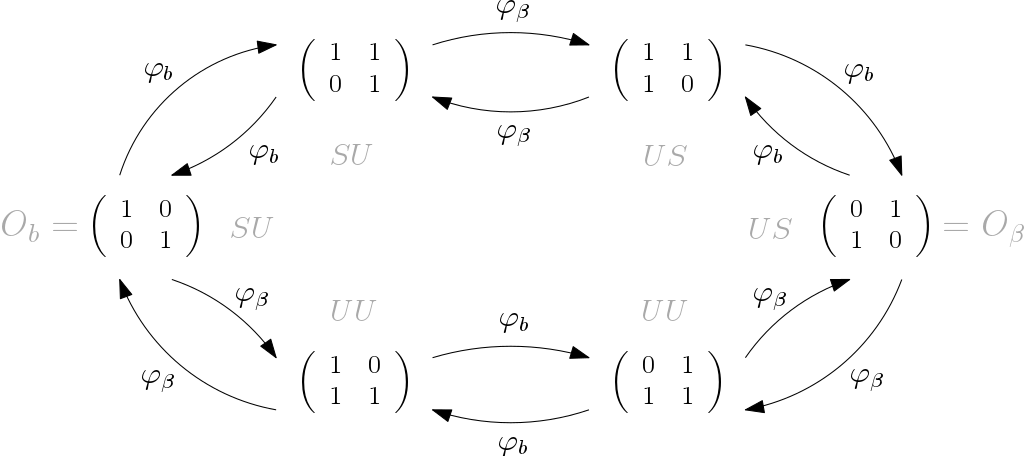}
\caption{The diagram captures the multiplication mod $2$ of matrices with the determinant $1$ by morphism matrices $M_b$ and $M_\beta$.
A directed edge labelled by $\varphi_b$ goes from matrix $M$ to matrix $M'$ if $M' = M_bM \mod 2$.
Analogously for the label $\varphi_\beta$.}  \label{Fig_Graph}
\end{figure}

\begin{rmrk}\label{exceptional}
Any standard Sturmian sequence $\uu$ with the directive sequence $\zz$ has bispecial prefixes of at least two types.
Indeed, by Proposition \ref{Prop_TypeDetermination} (or by the graph in Figure \ref{Fig_Graph}):

-- if $\zz$ has a prefix $b^{\ell}\beta $, $\ell \geq 1$, then 
the types of $w^{(0)}$ and $w^{(\ell)}$ are $SU$ and $US$, respectively;

-- if $\zz$ has a prefix $\beta^{2\ell}b$, $\ell\geq 1$, then the types of $w^{(0)}$, $w^{(1)}$ and $w^{(2\ell)}$ are $US$, $UU$ and $SU$, respectively;

-- if $\zz$ has a prefix $\beta^{2\ell-1}b$, $\ell\geq 1$, then the  types  of $w^{(0)}$ and  $w^{(2\ell-1)}$  are $US$ and $UU$,  respectively.
\medskip

It may happen that these bispecial prefixes are only of two types.
For example, if the directive sequence is $\zz = (\beta bb\beta)^\infty$, then for each $n \in \mathbb{N}$, the bispecial factor $w^{(2n)}$ is  of the type $US(1)$  and  the bispecial factor $w^{(2n+1)}$ is of the type $UU(2)$.
\end{rmrk}

We illustrate our results on the Rote sequence $\gg$ associated with the Fibonacci sequence $\ff$.

\begin{xmpl} \label{Exam_Fib1}
The Fibonacci sequence $\ff$ is fixed by the Fibonacci morphism
$F: 0 \to 01$, $1 \to 0$,  i.e.
$$\ff = 010010100100101001010010010100100101 \cdots \,.$$
Clearly, the sequence $\ff$ is a fixed point of the morphism $F^2$,  too.
Since $F^2 = \varphi_{b \beta}$, the Fibonacci sequence $\ff$ has the directive sequence $\zz = (b\beta)^\infty$, see Observation \ref{Obs_Epsilon}.
By Corollary \ref{composition}, the derivated sequence $\dd_{\ff}(w^{(2n)})$  has the directive sequence $(b\beta)^\infty$ and the derivated sequence $\dd_{\ff}(w^{(2n+1)})$ has the directive sequence $(\beta b)^\infty$.
It means that $\dd_{\ff}(w^{(2n)})$ is the Fibonacci sequence itself and $\dd_{\ff}(w^{(2n+1)})$ can be obtained from the Fibonacci sequence by exchange of letters $0\leftrightarrow 1$.
If we rewrite the derivated sequences into the alphabet $\{r,s\}$, where the most frequent letter is denoted by $r$ and the less frequent letter by $s$ (as required in Definition \ref{rsk}), we obtain only one derivated sequence $\dd$  (i.e. the Fibonacci sequence over the new alphabet):
$$\dd = r srrs  r srrs r r srs r r srs r  r srrs  r srrs r r srs \cdots\,.$$
Therefore, the derivated sequence to any prefix of the Fibonacci sequence is the Fibonacci sequence itself.

The types $\mathcal{T}^{(n)}$ of the bispecial prefixes of the Fibonacci sequence $\ff$ can be determined by Proposition \ref{Prop_TypeDetermination}, where the matrix products can be computed using Figure \ref{Fig_Graph}:

\begin{itemize}
\item $P^{(0)} = O_b =
\left(\begin{array}{cc}
1 & 0\\
0 & 1\\
\end{array} \right) \mod 2$ and $z_0z_1z_2 \cdots = b \beta b \beta \cdots$, thus $w^{(0)}$ has the type $\mathcal{T}^{(0)} = SU(1)$;
\item $P^{(1)} = M_bO_\beta =
\left(\begin{array}{cc}
1 & 1\\
1 & 0\\
\end{array} \right) \mod 2$ and $z_1z_2z_3 \cdots = \beta b \beta b \cdots$, thus $\mathcal{T}^{(1)} = US(1)$;
\item $P^{(2)} = M_b M_\beta O_b =
\left(\begin{array}{cc}
0 & 1\\
1 & 1\\
\end{array} \right) \mod 2$  and $z_2z_3z_4 \cdots = b \beta b \beta \cdots$, thus $\mathcal{T}^{(2)} = UU(1)$;
\item $P^{(3)} = M_b M_\beta M_b O_\beta =
\left(\begin{array}{cc}
1 & 0\\
0 & 1\\
\end{array} \right) \mod 2$ and $z_3z_4z_5 \cdots = \beta b \beta b \cdots$, thus $\mathcal{T}^{(3)} = SU(1)$;
\item $P^{(4)} = M_b M_\beta M_b M_\beta O_b =
\left(\begin{array}{cc}
1 & 1\\
1 & 0\\
\end{array} \right) \mod 2$ and $z_4z_5z_6 \cdots = b \beta b \beta \cdots$, thus $\mathcal{T}^{(4)} = US(1)$;
\item $P^{(5)} = M_b M_\beta M_b M_\beta M_b O_\beta =
\left(\begin{array}{cc}
0 & 1\\
1 & 1\\
\end{array} \right) \mod 2$ and $z_5z_6z_7 \cdots = \beta b \beta b \cdots$, thus $\mathcal{T}^{(5)} = UU(1)$.
\end{itemize}
By simple computations we get
$$M_b M_\beta M_b M_\beta M_b M_\beta =  I \mod 2\,, $$
where $I$ is the identity matrix.
Then we have
$$P^{(6)} = M_b M_\beta M_b M_\beta M_b M_\beta O_b = O_b \mod 2\, .$$
We have also $z_6z_7z_8 \cdots = b \beta b \beta \cdots = z_0z_1z_2 \cdots$. So we can conclude that $w^{(6)}$ has the same type as $w^{(0)}$, which is $SU(1)$. Similarly, $w^{(7)}$ has the same type as $w^{(1)}$ etc.

Now we use Corollary \ref{theSame} to describe the derivated sequences of the Rote sequence 
$$\gg = 001110011100011000110001110011100011 \cdots $$
associated with the Fibonacci sequence $\ff$. 
Since all derivated sequences of $\ff$ are the same and $\ff$ has three distinct types of bispecial prefixes, there are exactly three distinct derivated sequences of $\gg$: $\dd_{\gg}(x^{(0)})$, $\dd_{\gg}(x^{(1)})$ and $\dd_{\gg}(x^{(2)})$.

Finally, we show how to construct these derivated sequences $\dd_{\gg}(x^{(0)})$, $\dd_{\gg}(x^{(1)})$ and $\dd_{\gg}(x^{(2)})$.
Since the type of $x^{(0)}$ is $SU(1)$, the return words $A, B, C$ to the prefix $x^{(0)}$ correspond to the Sturmian factors $r$, $sr^2s$ and $srs$, respectively, where $r, s$ are the return words to $w^{(0)}$, see Theorem \ref{Thm_ReturnWords}.
Thus we have to decompose the sequence $\dd \in \{r, s\}^\N$
$$
\dd = r srrs  r srrs r r srs r r srs r  r srrs  r srrs r r srs \cdots
$$
onto blocks  $ r$, $sr^2s$ and $ srs$. 
The order of these blocks gives us the derivated sequence of $\gg$ to $x^{(0)}$
$$
\dd_{\gg}(x^{(0)}) = AB ABAACAACA AB ABAAC \cdots \,.
$$
The type of $w^{(1)}$ is $US(1)$, so the return words $A, B, C$ to $x^{(1)}$  correspond to the Sturmian factors $rr$, $rsr$, $s$, respectively.
So we decompose $\dd$ onto blocks $ rr$, $rsr$, $s$ and we get
$$\dd_{\gg}(x^{(1)}) = BBCACAC BBCACAC B BBC\cdots\,.$$
The type of $w^{(2)}$ is $UU(1)$, so the return words $A, B, C$ to $x^{(2)}$  correspond to the Sturmian factors $rr$, $rs$, $sr$, respectively.
So we decompose $\dd$ onto blocks $rr$, $rs$, $ sr$ and we get
$$\dd_{\gg}(x^{(2)}) = BACC BACCB BACB BACB B\cdots\,.$$

As explained in Section \ref{Sec_DerivatedSequences}, the derivated sequences $\dd_{\gg}(x^{(0)})$, $\dd_{\gg}(x^{(1)})$ and $\dd_{\gg}(x^{(2)})$ are 3iet sequences.
We can find the parameters of their interval exchange transformations using Proposition \ref{Prop_ThreeExchange}.
The Fibonacci sequence has the slope $\alpha = \frac{1}{\tau}$ and the intercept $\rho = 1- \frac{1}{\tau} = 2- \tau$, where $\tau$ denotes the golden ratio $(1+\sqrt{5}) / 2$.
Thus these derivated sequences are 3iet sequences coding the intercept $2- \tau$ under the three interval exchange transformation $T$ with the parameters $\beta$, $\gamma$ and the permutation $\pi$ as follows:
\begin{itemize}
\item for $\dd_{\gg}(x^{(0)})$ the parameters are $\beta = \frac{1}{\tau}$, $\gamma = \frac{2}{\tau} -1$ and $\pi = (3,2,1)$;
\item for $\dd_{\gg}(x^{(1)})$ the parameters are $\beta = \frac{2}{\tau} -1$, $\gamma = 2- \tau$ and $\pi = (3,2,1)$;
\item for $\dd_{\gg}(x^{(2)})$ the parameters are $\beta = \frac{2}{\tau} - 1$, $\gamma = 2- \tau$ and $\pi = (2,3,1)$.
\end{itemize}
\end{xmpl}

\section{Derivated sequences of substitutive complementary symmetric Rote sequences} \label{S_Substitutive}

The aim of this section is to decide when a Rote sequence associated with a standard Sturmian sequence is primitive substitutive, i.e. it is a morphic image of a fixed point of a primitive morphism.
First we explain why a Rote sequence $\vv$ cannot be purely primitive substitutive, i.e. cannot be fixed by a primitive morphism.

\begin{lmm}\label{notFixed}
Let $\vv$ be a Rote sequence associated with a standard Sturmian sequence $\uu = \S(\vv)$. 
Then $\vv$ is not a fixed point of a primitive morphism.
\end{lmm}

\begin{proof}
Let us assume that $\vv$ is fixed by a primitive  morphism $\varphi$. 
Then the vector of letter frequencies $(\rho_0, \rho_1)^\top$  is an eigenvector to the dominant eigenvalue $\Lambda$ of the matrix $M_\varphi\in \mathbb{N}^{2\times 2}$. 
As the language of $\vv$ is closed under the exchange of letters $0\leftrightarrow1$, the vector of frequencies is $(\rho_0, \rho_1)^\top = (\tfrac12, \tfrac12)^\top$, see \cite{Que87}. 
Since all entries of the primitive matrix $M_\varphi$  are integer, an eigenvalue to a rational eigenvector is an integer number.
Moreover, by the Perron-Frobenius theorem the dominant eigenvalue of any primitive matrix with entries in $ \mathbb{N}$  is bigger then $1$, i.e. $\Lambda >1$.
The second eigenvalue $\Lambda'$ (i.e. the other zero of the quadratic characteristic polynomial of $M_\varphi$) is integer, too.

Let $x$ be a prefix of $\vv$ and $\dd_{\vv}(x)$ be the derivated sequence of $\vv$ to the prefix $x$.  
Let us assume that $\dd_{\vv}(x)$ is fixed by a primitive morphism $\psi$.
By Proposition \ref{Prop_ThreeExchange}, the derivated sequence $\dd_{\vv}(x)$ is a ternary sequence coding a three interval exchange transformation. 
The letter frequencies in any 3iet sequence are given by the lengths of the corresponding subintervals.
The lengths of three subintervals described in Proposition \ref{Prop_ThreeExchange} are irrational as the slope  $\alpha$   of a  Sturmian sequence is irrational. Therefore, the vector of frequencies  and consequently the dominant eigenvalue of the matrix $M_{\psi}$  is irrational as well.

For a sequence fixed by a primitive morphism $\eta$, Durand in \cite{Dur08} proved  that any its derivated sequence is fixed by some morphism, say $\xi$, and each  eigenvalue $\lambda$ of  $M_\xi$ either belongs to the spectrum of $M_\eta$ or its modulus $|\lambda|$ belongs to $\{0,1\}$.

Applying this result to the morphisms $\varphi$ fixing the Rote sequence $\vv$ and  the morphism $\psi$ fixing its derivated sequence $\dd_{\vv}(x)$ we get that the spectrum of $M_{\psi}$ is a subset of $\{\Lambda, \Lambda', 0\}\cup\{y\in \mathbb{C} :|y| = 1\}$. 
Thus the dominant eigenvalue of the matrix  $M_\psi$  (which is by the Perron-Frobenius theorem bigger than 1)  cannot be irrational. This is a contradiction.
\end{proof}

Despite the previous lemma, we will show in Theorem \ref{Substitutive} that a Rote sequence is primitive substitutive whenever the associated Sturmian sequence is primitive substitutive. 
For the proof we need to study the periodicity of the types of bispecial prefixes in Sturmian sequences.

\begin{prpstn} \label{Prop_OtherOccurrence}
Let $\uu$ be a standard Sturmian sequence with an  eventually periodic directive sequence $\zz$ with a period $Q$. 
Then there exists $q \in \{1, 2 ,3\}$ such that the sequence $\bigl(P^{(n)}\bigr)_{n\in \mathbb{N}}$ is eventually periodic with a period $qQ$.

\end{prpstn}
\begin{proof} Let $p$ be a preperiod of the directive sequence  $\zz= z_0z_1z_2 \cdots$.  
We denote
$$H= M_{z_{p}} M_{z_{p+1}}  \cdots M_{ z_{p + Q-1}} \mod 2\,.$$   
The matrix $H$ belongs to the set of matrices displayed in \eqref{Set_matrices}.
One can easily verify that
\begin{itemize}
\item[i)] if $H \in \left\{
\left(\begin{array}{cc}
1 & 1\\
0 & 1\\
\end{array} \right),
\left(\begin{array}{cc}
1 & 0\\
1 & 1\\
\end{array} \right),
\left(\begin{array}{cc}
0 & 1\\
1 & 0\\
\end{array} \right)
\right\}$, then $H^2 = I \mod 2$;
\item[ii)] if $H \in \left\{
\left(\begin{array}{cc}
0 & 1\\
1 & 1\\
\end{array} \right),
 \left(\begin{array}{cc}
1 & 1\\
1 & 0\\
\end{array} \right)
\right\}$, then $H^3 = I \mod 2$.
\end{itemize}
Let $q$ be the smallest positive integer such that $H^q=I \mod 2$, obviously $q \in \{1,2,3\}$.
To conclude the proof we show that the sequence  $\bigl(P^{(n)}\bigr)_{n\in \mathbb{N}}$ has a preperiod $p$ and a period $qQ$. 
Let $n\geq p$ and $m=n+qQ$.
By Proposition \ref{Prop_TypeDetermination} and the fact that $z_i =z_{i+qQ}$ for any $i\geq p$ we can write
\begin{align*}
P^{(n)} &= M_{z_0}\cdots M_{z_{n-1}} O_{z_n}=\Bigl(M_{z_0}\cdots M_{z_{p-1}}\Bigr)\Bigl(M_{z_{p}}\cdots M_{z_{n-1}}\Bigr)O_{z_n} \,,\\
P^{(n+qQ)} = P^{(m)} &=M_{z_0}\cdots M_{z_{m-1}} O_{z_m} = \Bigl( M_{z_0} \cdots M_{z_{p-1}}\Bigr)\  H^q \ \Bigl(M_{z_{p}} \cdots M_{z_{n -1}}\bigr) O_{z_n} = P^{(n)} \,.
\end{align*}
\end{proof}

\begin{thrm} \label{Substitutive}
Let $\vv$ be a Rote sequence associated with a standard Sturmian sequence $\uu = \S(\vv) $. 
Then the Rote sequence $\vv$ is primitive substitutive if and only if the Sturmian sequence $\uu$ is primitive substitutive.
\end{thrm}

\begin{proof}  
The proof is based on two results:

-- A standard Sturmian sequence is primitive substitutive if and only if its directive sequence $\zz \in \{b, \beta\}^\mathbb{N}$ is  eventually periodic, see Observation \ref{Obs_Epsilon}.

-- A sequence is primitive substitutive if and only if it has finitely many derivated sequences,  see Theorem \ref{Durand}.
\medskip

Let us assume that $\uu$ is primitive substitutive and $Q$ is a  period of  its directive sequence $\zz = z_0z_1z_2\cdots$.
By Corollary \ref{composition}, the sequence $\bigl(\dd_\uu(w^{(n)})\bigr)_{n\in \mathbb{N}}$  is eventually periodic with the  same  period $Q$.

Remind that the type $\mathcal{T}^{(n)}$  of a bispecial  prefix $ w^{(n)}$ is determined by the bottom row of the matrix  $P^{(n)}$ and by the length of the maximal monochromatic prefix of the sequence $z_nz_{n+1}z_{n+2}\cdots$.  
By Proposition \ref{Prop_OtherOccurrence}, the sequence $\bigl(\mathcal{T}^{(n)}\bigr)_{n\in \mathbb{N}}$ is eventually periodic with the period $qQ$.

It implies that the sequence of pairs  $\bigl(\mathcal{T}^{(n)},\dd_\uu(w^{(n)})\bigr)_{n\in \mathbb{N}}$ is eventually periodic with the period $qQ$, too.
Then by Corollary \ref{theSame}, the Rote sequence $\vv$ has only finitely many derivated sequences and thus the sequence $\vv$ is primitive substitutive.

On the other hand, if  $\vv$ is primitive substitutive, $\vv$ has only finitely many derivated sequences. 
Then by Corollary \ref{theSame} the associated Sturmian sequence $\uu$ has only finitely many derivated sequences and thus $\uu$ is primitive substitutive.
\end{proof}

\begin{crllr} \label{Thm_Number}
Let $\vv$ be a Rote sequence associated with a standard Sturmian sequence $\uu = \S(\vv) $ fixed by a primitive morphism $\varphi_z$, where $z\in \{b, \beta\}^+$. 
Then $\vv$ has at most $3 |z|$ distinct derivated sequences to its non-empty prefixes and each of them is fixed by a primitive morphism over a ternary alphabet.
\end{crllr}

\begin{proof}
Using the notation from the proofs of Proposition \ref{Prop_OtherOccurrence} and Theorem \ref{Substitutive}, the sequence of pairs  $\bigl(\mathcal{T}^{(n)},\dd_\uu(w^{(n)})\bigr)_{n\in \mathbb{N}}$ has a preperiod $p$ and a period $qQ$, where $q\in\{1,2,3\}$. 
Since now the directive sequence $\zz = z^{\infty}$ is purely periodic, we can choose $p=0$ and $Q=|z|$.
Each pair $\bigl(\mathcal{T}^{(n)},\dd_\uu(w^{(n)})\bigr)$ uniquely determines a derivated sequence in $\vv$, and thus there are at most $qQ \leq 3|z|$ distinct derivated sequences to non-empty prefixes of $\vv$.

For the second part of the statement, it suffices to apply Durand's result from \cite{Dur98}: if a sequence $\dd$  is the derivated sequence to two distinct prefixes, then $\dd$ is a fixed point of some morphism.
\end{proof}

Let us stress  that the previous corollary does not speak about the derivated sequence to the prefix $\varepsilon$.  
In this case the derivated sequence $\dd_{\vv}(\varepsilon)$ is the sequence $\vv$ itself and by Lemma \ref{notFixed} it is not a fixed point of any primitive morphism.

\begin{rmrk}\label{smallerPeriod}  
The number of derivated sequences of a Rote sequence $\vv$ may be smaller than the value $3|z|$ announced in Corollary \ref{Thm_Number} and also smaller than the value $qQ$ found in the proof.  
There are two reasons which may diminish the number:

i) The period  $Q$  of the sequence $\bigl(\dd_\uu(w^{(n)})\bigr)_{n\in \mathbb{N}}$ comes from the period $Q$ of the directive sequence $\zz = z^{\infty}$, where the word $z \in \{b,\beta\}^+$  describes the Sturmian morphism $\varphi_z$. 
If the word $z$  is not primitive, i.e. $z = y^m$ for some $y \in \{b,\beta\}^+$ and $m \in \mathbb{N}, m\geq 2$, we can replace $|z|$ by the smaller number $|y|$. 
But even if $z$  is primitive, the minimal period of $\bigl(\dd_\uu(w^{(n)})\bigr)_{n\in \mathbb{N}}$ may be smaller. It happens for example in the Fibonacci case, where we consider the morphism $\varphi_{b\beta}$, i.e. $z=b\beta$, see Example \ref{Exam_Fib1}.

ii) The sequence of matrices $\bigl(P^{(n)}\bigr)_{n\in \mathbb{N}}$ has the guaranteed period $qQ$. 
But since the type $\mathcal{T}^{(n)}$ is determined only by the bottom row of the matrix $P^{(n)}$, it may also happen that $\mathcal{T}^{(m)} = \mathcal{T}^{(n)}$ for a pair $n, m \in \N$, $n < m < n+ qQ$.  
\end{rmrk}

By the proof of Corollary \ref{Thm_Number}, if two distinct prefixes of $\vv$ has the same derivated sequence, then this common derivated sequence is fixed by some morphism.  
Durand in \cite{Dur98} provided a construction of this fixing morphism.

\subsection*{Durand's construction of fixing morphisms}

Here we remind the construction only for the case when each non-empty prefix $x$ of the sequence $\vv$ has exactly three return words in $\vv$.
We assume:
\begin{itemize}
\item[--] $x$ and $x'$ are prefixes of a sequence $\vv \in \mathcal{A}^\mathbb{N}$ such that $|x| <|x'|$;
\item[--] $A, B,C\in \mathcal{A}^+$  are the  return words to $x$  and $A',B', C' \in \mathcal{A}^+$ are the return words to $x'$;
\item[--] the derivated sequences $\dd_{\vv}(x)$ over $\{A,B,C\}$ and  $\dd_{\vv}(x')$ over $\{A',B',C'\}$ satisfy
$$\dd_{\vv}(x')=\pi\bigl( \dd_{\vv}(x)\bigr)  \,,\ \ \ \text{where} \  \pi \ \text{ is the projection $A\to A',  B \to B', C\to C'$}.$$
\end{itemize}
Since  $x$ is a prefix of $x'$, the words $A',B', C'$ are concatenations of the words $A, B, C$ and we can write $A', B', C' \in \{A,B,C\}^+$.  
Thus one can find  the words $w_A, w_B,  w_C \in \{A,B,C\}^+ $ such that  $A' = w_A$, $B' = w_B$ and $C'=w_C$. 
Then the derivated sequence $\dd_{\vv}(x)$  is fixed by the  morphism $\sigma : \{A,B,C\}^* \to  \{A,B,C\}^*$  defined by
$$\sigma(A) = w_A, \ \ \sigma(B) = w_B, \ \ \sigma(C) = w_C\,.$$
\medskip

If $\vv$ is a Rote sequence associated with a standard Sturmian sequence $\uu$, the Durand's construction can be transformed into the manipulation with the factors of the Sturmian sequence $\uu$ instead of factors of the Rote sequence $\vv$.

Let $x$, $x'$ be two non-empty bispecial prefixes of $\vv$ with the same derivated sequence, i.e. $\dd_{\vv}(x) = \dd_{\vv}(x')$, and let $|x| < |x'|$.    
By Corollary \ref{Coro_Bispecials}, $w:= \S(x)$ and $w':=\S(x')$ are  bispecial  prefixes of $\uu$. 
By Corollary \ref{theSame}, $w$ and $w'$ have the same type and the same derivated  sequence of $\uu$.
Denote by $r,s$ the most frequent and the less frequent return word to $w$ in $\uu$ and analogously denote the return words $r',s'$ to the prefix $w'$.

Let  $A,B,C$  and $A', B',C'$  be the return words to   $x$ and $x'$  in $\vv$, respectively.  
Put
 $$a:=\S(A0)\,,  \ \ b:=\S(B0)\,, \ \ c:=\S(C0) \qquad  \text{and} \qquad a':=\S(A'0)\,,  \ \ b':=\S(B'0)\,,  \ \  c':=\S(C'0) \,.$$
Theorem \ref{Thm_ReturnWords} implies $a,b,c\in \{r,s\}^+$ and  $a',b',c'\in \{r',s'\}^+$. 
Recall that $A',B', C' \in \{A,B,C\}^+$ and as the types of bispecial prefixes $w$ and $w'$ are the same, necessarily  $a',b',c' \in\{a,b,c\}^+ $ as well. 
Therefore we can find the words $w_a,w_b,  w_c \in \{a,b,c\}^+$ such that $a'= w_a,  b'= w_b,  c'= w_c$. 
Then the desired morphism $\sigma$ fixing the derivated sequence $\dd_{\vv}(x)$ is defined by $A\to \pi(w_a), B\to \pi(w_b), C\to \pi(w_c)$,  where $\pi$ is the projection $a\to A, b\to B, c\to C$.

\medskip

Moreover, if $\uu$ is a fixed point of a primitive Sturmian morphism, then all its derivated sequences are fixed by some primitive Sturmian morphisms as well, see Corollary \ref{composition}. 
Thus the search for the morphism $\sigma$ can be simplified as $r', s' \in \{r, s\}^*$ are images of $r, s$ under a Sturmian morphism over the alphabet  $\{r,s\}$.

\begin{xmpl} \label{Exam_Fib2}
In Example \ref{Exam_Fib1} we have showed that the Rote sequence $\gg$ associated with the Fibonacci sequence $\ff$ has three derivated sequences $\dd_\gg(x^{(0)})$, $\dd_\gg(x^{(1)})$ and $\dd_\gg(x^{(2)})$.  
Now we find their fixing morphisms $\sigma_0$, $\sigma_1$, and $\sigma_2$, respectively.

Let us start with $\dd_\gg(x^{(0)})$. 
We have $\dd_\gg(x^{(0)}) = \dd_\gg(x^{(3)})$ and so $\dd_\ff(w^{(0)}) = \dd_\ff(w^{(3)})$. 
The return words to $w^{(0)}$ are $r = 0$, $s = 1$ and the return words to $w^{(3)}$ are $r' = 01001$, $s' = 010$.
Both $w^{(0)}$ and $w^{(3)}$ have the same type $SU(1)$.
Thus the return words to $x^{(0)}$ correspond to blocks
$$a:= \S(A0)=r = 0\,, \ \ \ b:= \S(B0)=srrs = 1001\,, \ \ \ c:= \S(C0)=srs = 101\, $$
and the return words to $x^{(3)}$ correspond to blocks
$$a'=r'= 01001\,, \ \ \ \ b'= s'r'r's' = 010 01001 01001 010\,, \ \ \ \ c'=s'r's' = 010 01001 010\,, $$
where we denote $a':= \S(A'0)$,  $ b':= \S(B'0)$ and $c':= \S(C'0)$.

If we decompose $a' = 01001$ into $a = 0$, $b = 1001$, $c = 101$, we get $a' = ab$. 
Similarly $b' = abaacaaca$ and $c' = abaaca$. 
So the fixing morphism $\sigma_0$ is defined as follows:
$$
\sigma_0:
\left\{
\begin{aligned}
 A &\to AB \\
B &\to ABAACAACA \\
C &\to ABAACA\,
\end{aligned}
\right.\,.
$$

Equivalently, we can also find $\sigma_0$ without knowledge of the return words since the Fibonacci sequence $\ff$ is the fixed point of the Fibonacci morphism. 
As shown in Example \ref{Exam_Fib1}, all derivated sequences of $\ff$ are equal to the Fibonacci sequence over the alphabet $\{r, s\}$ and so they are fixed by the morphism $\psi: r\to rs, s\to r$.
Since the sequence of types $(\mathcal{T}^{(n)})_{n \in \N}$ has a period $3$, the return words $r'$ and $s'$ satisfy
$r' = \psi^3(r)$ and $s'=\psi^3(s)$. 
Therefore, it is enough to factorize $a'= \psi^3(r)$,  $ b'=\psi^3(srrs)$ and $ c'=\psi^3(srs)$ into the blocks $a=r$, $b=srrs$ and $c=srs$. 
We get
$$\begin{array}{lllllllll}
a'&=& r'& =& \psi^3(r)& =& rsrrs&=& ab,\\
b'&=&s'r'r's'&=&\psi^3(srrs)& =& rsrrsrrsrsrrsrsr& =& abaacaaca,\\
c'&=&s'r's'&=&\psi^3(srs) &= &rsrrsrrsrsr & = &abaaca\,.
\end{array}$$
It exactly  corresponds to the morphism $\sigma_0$.

\medskip

Now we find a morphism $\sigma_1$ fixing the derivated sequence of $\gg$ to the prefix $x^{(1)}$. 
As the bispecial prefix $w^{(1)}$ has the type $ US(1)$, we work with the blocks $a=rr$, $b = rsr$ and $c=s$. 
We apply the same method and we get
$$\begin{array}{lllllllll}
a'&=& r'r'& =& \psi^3(rr)& =& rsrrsrsrrs&=& bbcac,\\
b'&=&r's'r'&=&\psi^3(rsr)& =& rsrrsrsrrsrrs & =& bbcacac,\\
c'&=&s'&=&\psi^3(s) &= &rsr & = &b\,.
\end{array}$$
Thus the derivated sequence $\dd_{\gg}(x^{(1)})$ is fixed by the morphism
$$
\sigma_1:
\left\{
\begin{aligned}
 A &\to BBCAC \\
B &\to BBCACAC \\
C &\to B\,
\end{aligned}
\right.\,.
$$

Finally, as $w^{(2)}$ has the type $UU(1)$, we work with the blocks $a =rr$, $b=rs$ and $c=sr$ and we get 

$$\begin{array}{lllllllll}
a'&=& r'r'& =& \psi^3(rr)& =& rsrrsrsrrs&=& baccb,\\
b'&=&r's'&=&\psi^3(rs)& =& rsrrsrsr & =& bacc,\\
c'&=&s'r'&=&\psi^3(sr) &= &rsrrsrrs & = &bacb\,.
\end{array}$$
Therefore, the morphism fixing the derivated sequence $\dd_{\gg}(x^{(2)})$ is
$$
\sigma_2:
\left\{
\begin{aligned}
 A &\to BACCB \\
B &\to BACC\\
C &\to BACB\,
\end{aligned}
\right.\,.
$$
\end{xmpl}
\medskip

We finish this section by an algorithm for finding the morphisms fixing the derivated sequences of Rote sequences.
To simplify the notation we use the cyclic shift operation  ${\rm cyc}: \{b, \beta\}^+\to \{b, \beta\}^+$ defined by
$${\rm cyc}(z_0z_1\cdots z_{Q-1}) = z_1z_2\cdots z_{Q-1}z_0\,.$$
By Corollary \ref{composition}, if $\zz = z^\infty $, where $z= z_0z_1\cdots z_{Q-1} \in \{b,\beta\}^+$, is a directive sequence of a Sturmian sequence $\uu$, then the derivated sequence  $\dd_\uu(w^{(n)})$ to the $n^{th}$ bispecial prefix of $\uu$ has the directive sequence $\zz^{(n)} =\bigl({\rm cyc}^n( z)\bigr)^\infty$, and thus $\dd_\uu(w^{(n)})$ is fixed by the morphism $\varphi_{{\rm cyc}^n( z)}$.

\begin{lgrthm} \label{Algo_FixingMorphism}
\smallskip
\mbox{} \newline
\noindent Input: $z \in \{b, \beta\}^* $ such that both $b$ and $\beta$ occur in $z$.
\smallskip

\noindent Output: the list of morphisms over $\{A,B,C\}$  fixing the derivated sequences of the Rote sequence $\vv$ associated with the fixed point of the  morphism $\varphi_z$.
\smallskip

1. Denote  $Q=|z|$, $\zz = z_0z_1z_2\cdots = z^\infty$  and $H =M_z$.
\smallskip

2. Find the minimal $q\in\{1,2,3\}$ such that $H^q =I \mod 2$.
\smallskip

3. For $i=0,1,2,\ldots,  qQ-1$ do:

\begin{enumerate}
\item Compute $P^{(i)} = M_{z_0}M_{z_1}\cdots M_{z_{i-1}}O_{z_i} \mod 2$ and determine the type $\mathcal{T}^{(i)}$.
\item Rewrite $\varphi_{{\rm cyc}^i( z)}$ into the alphabet $\{r,s\}$ by the rule 
$0 \to r, 1 \to s$ if the first letter of ${\rm cyc}^i( z)$ is $b$ and by the rule $0 \to s, 1 \to r$ otherwise  
Denote this morphism $\psi$.
\item Define  $a,b,c \in \{r,s\}^+$ according to the type $\mathcal{T}^{(i)}$.
\item Compute  $ \psi^q(a)$, $ \psi^q(b)$, $\psi^q(c)$.
\item Decompose the words $ \psi^q(a)$, $ \psi^q(b)$, $\psi^q(c)$ into the blocks $a,b,c$, i.e.,
find $w_a, w_b, w_c \in \{a,b,c\}^+$  \  such that  \
 $ \psi^q(a) = w_a, \  \psi^q(b) = w_b,  \    \psi^q(c) = w_c$.
\item Put into the list the morphism $\sigma_i:$ $A\to \pi(w_a)$, $B\to \pi(w_b)$, $C\to \pi(w_c)$, where the projection $\pi$ rewrites $a\to A$, $b\to B$, $c\to C$.
\end{enumerate}
\end{lgrthm}

\begin{xmpl}[Example \ref{Exam_BBetaB4} continued] \label{Exam_BBetaB5}
Our aim is to describe the derivated sequences of the Rote sequence associated with the fixed point of $\varphi_{b \beta b}$. 
Each its derivated sequence is fixed by a primitive morphism which we find with Algorithm \ref{Algo_FixingMorphism} to the input $z=b\beta b$.

1. $Q=|b\beta b| = 3$,  $\zz = (b\beta b)^\infty$ and  $H=M_b M_\beta M_b =
\left(\begin{array}{cc}
0 & 1\\
1 & 0\\
\end{array} \right) $.

2. $q=2$ as $H^2 = I \mod 2$.
\smallskip

3. For $i=0,1, ..., 5$  (we illustrate the step only for $i=2$) do:

\begin{enumerate}
\item  $P^{(2)} = M_b M_\beta O_b =
\left(\begin{array}{cc}
0 & 1\\
1 & 1\\
\end{array} \right) \mod 2$ \  and $z_2z_3z_4 \cdots = bb\beta \cdots$, thus  \ \  $\mathcal{T}^{(2)}=UU(2)$.
\item  $\varphi_{{\rm cyc}^2(b\beta b)} = \varphi_{b b \beta} : \left\{\begin{array}{l}0\to 0010\\1\to 001
\end{array}\right.$ \quad  and thus  \ \ $\psi : \left\{\begin{array}{l}r\to rrsr\\s\to rrs
\end{array}\right.$.
\medskip
\item  $a=rr$ , $b=rs$,  $c=sr$.
\medskip
\item   Since  \ $\psi^2(r) =  rrsr rrsr   rrs   rrsr$ \ and \ $\psi^2(s) =rrsr rrsr   rrs $, \  we have

\medskip

$\psi^2(a) =\psi^2(rr) =  rrsr rrsr   rrs   rrsr  rrsr rrsr   rrs   rrsr   $;

\medskip

$\psi^2(b) =\psi^2(rs)  =  rrsr rrsr   rrs   rrsr  rrsr rrsr   rrs  $;
\medskip

$\psi^2(c) = \psi^2(sr)  =    rrsr rrsr   rrs     rrsr rrsr   rrs   rrsr  $.
\medskip

\item   $\psi^2(a) =\underbrace{rr}_a\underbrace{sr}_c \underbrace{rr}_a \underbrace{sr}_c   \underbrace{rr}_a\underbrace{sr}_c \underbrace{rs}_b \underbrace{rr}_a\underbrace{rs}_b\underbrace{rr}_a\underbrace{rs}_b \underbrace{rr}_a\underbrace{rs}_b \underbrace{rr}_a\underbrace{sr}_c $,

thus $w_a = acacacbabababac$;

\medskip

\noindent $\psi^2(b) =\underbrace{rr}_a\underbrace{sr}_c \underbrace{rr}_a \underbrace{sr}_c   \underbrace{rr}_a\underbrace{sr}_c \underbrace{rs}_b \underbrace{rr}_a\underbrace{rs}_b\underbrace{rr}_a\underbrace{rs}_b \underbrace{rr}_a\underbrace{rs}_b$,

thus $w_b=acacacbababab$;

\medskip

\noindent $\psi^2(c)=\underbrace{rr}_a\underbrace{sr}_c \underbrace{rr}_a \underbrace{sr}_c   \underbrace{rr}_a\underbrace{sr}_c \underbrace{rs}_b \underbrace{rr}_a\underbrace{rs}_b\underbrace{rr}_a\underbrace{rs}_b \underbrace{rr}_a\underbrace{sr}_c $,

thus $w_c = acacacbababac$.

\medskip

\item We add to the list the morphism
$$
\sigma_2:
\left\{
\begin{aligned}
A &\to ACACACBABABABAC \\
B &\to ACACACBABABAB\\
C &\to ACACACBABABAC\,
\end{aligned}
\right.\,.
$$
\end{enumerate}
\end{xmpl}

\section{Original Rote's construction and morphisms on four versus three letter alphabet} \label{otherview}

In the original Rote's paper \cite{Ro94}, the author also construct Rote sequences as projections of fixed points on a four letter alphabet.  
Let us define the morphism $\xi$ and the projection $\pi$ as follows:
$$
\xi:
\left\{
\begin{aligned}
1 &\to 13 \\
2 &\to 24 \\
3 &\to 241\\
4 &\to 132\,
\end{aligned}
\right.\,
\quad \quad  \text{and} \quad \quad
\pi:
\left\{
\begin{aligned}
1 &\to 0\\
2 &\to 1 \\
3 &\to  0\\
4 &\to  1\,
\end{aligned}
\right.\,.
$$
The morphism $\xi$ has two fixed points. We take its fixed point  ${\bf q}$ starting with the letter $1$, i.e.
$${\bf q}=13 24 124 13 213 24132 13241 2413 241  241321324124132412413\cdots\,,$$
and we apply the projection $\pi$ to construct the complementary symmetric Rote sequence:
$${\bf r} = \pi({\bf q})=
00110110010011001001101100110110010011011001101100 \cdots \,.$$
Using the operation $\S$ (see Definition \ref{S}), we obtain  the associated Sturmian sequence 
$${\bf s}= \S\bigl({\bf r}\bigr)=
0101101011010101101011010101101011010110101011010 \cdots\,.$$

This example was in fact the beginning of this work.  
We notice that if we add the first letter $1$ to the associated Sturmian sequence $\bf s$, we get the Sturmian sequence
$${\bf u}=1{\bf s}=
101011010110101011010110101011010110101101010110101 \cdots$$
which is also fixed by a morphism, namely by the  morphism:
$$
\psi :
\left\{
\begin{aligned}
0 &\to 101 \\
1 &\to 10\,
\end{aligned}
\right.\,.
$$
The question is how to link this Rote's example to our construction.

In fact, the morphism $\psi$ is a standard Sturmian morphism with the decomposition
$\psi = \varphi_{\beta}\varphi_b E$. 
Thus we can use our techniques to find the return words and the derivated sequence to the prefix $x = 01$ of the Rote sequence
$$\vv=
0110010011011001101100100110010011011001001100100110 \cdots\,.$$
associated with the sequence $\bf u$. 
We obtain the return words $A=0110$, $B=010$ and $C= 011$. The derivated sequence $\dd_{\vv}(x)$ of $\vv$ to the prefix $x=01$ is fixed by the morphism
$$
\sigma:
\left\{
\begin{aligned}
A &\to ABC\\
B &\to AC \\
C &\to AB\,
\end{aligned}
\right.\,.
$$
The Rote sequence $\vv$ is clearly the image of the fixed point of $\sigma$, i.e. the derivated sequence $\dd_{\vv}(x)$, under the projection $\rho$ defined as:
$$
\rho(A)=0110\,,\quad \rho(B)=010\,,\quad \rho(C)= 011\,.
$$

As ${\bf u} = 1 {\bf s}$, their associated Rote sequences $\vv$ and ${\bf r}$ are tied by
$$ \vv = 011001001101100110110010011 \cdots =  
\overline{100110110010011001001101100} \cdots =  \overline{1{\bf r}} \,.
$$ 
Moreover, all return words to $01$ in $\vv$ obviously start with $0$. 
Thus the original Rote sequence ${\bf r}$ is the image of the fixed point of $\sigma$ under the projection $\rho'$ defined as
$$
\rho'(A)=\overline{1100}=0011,\quad \rho'(B)=\overline{100}=011,\quad \rho'(C)=\overline{110}=001.
$$
In other words, ${\bf r} = \pi({\bf q}) = \rho'(\dd_{\vv}(x))\,.$

The morphism $\xi$ on a four letter alphabet can be recovered as follows.
We write the images of the projection $\rho'$ as suitable projections by $\pi$, more precisely
$$\rho'(A)=0011=\pi(1324),\quad \rho'(B)=011= \pi(124),\quad \rho'(C)=001=\pi(132)\,.$$
Then the projections of the images of the morphism $\sigma$ can be expressed as
\begin{align*}
\rho'(\sigma(A)) &= \rho'(ABC)= \pi(1324124132) = \pi(\xi(1)\xi(3)\xi(2)\xi(4))\,, \\
\rho'(\sigma(B)) &= \rho'(AC)=\pi(1324132) = \pi(\xi(1)\xi(2)\xi(4))\,, \\
\rho'(\sigma(C)) &= \rho'(AB)=\pi(1324124) = \pi(\xi(1)\xi(3)\xi(2))\,,
\end{align*}
and by recoding to the suffix code $\{13, 24, 241, 132\}$ we get the morphism $\xi$ as above. 

\medskip

Finally, let us mention that we are able to generalize the original example as follows. For $n \in \N$ consider the morphism $\xi_n$ defined as
$$
\xi_n:
\left\{
\begin{aligned}
1 &\to 13 \\
2 &\to 24 \\
3 &\to (2413)^n241\\
4 &\to (1324)^n132\,
\end{aligned}
\right.\,.
$$
Then the $\pi$ projection of its fixed point $\bf r$ starting with $1$ can be seen as the $\rho'$ projection of the fixed point of the morphism $\sigma_n$, where $\sigma_n$ and $\rho'$ are the following: 
$$
\sigma_n:
\left\{
\begin{aligned}
A &\to A^nABA^nC\\
B &\to A^n AC \\
C &\to A^n AB\,
\end{aligned}
\right.\,
\quad \quad
\text{and} 
\quad \quad
\rho':
\left\{
\begin{aligned}
A &\to 0011\\
B &\to 011 \\
C &\to 001\,
\end{aligned}
\right.\,.
$$

Nevertheless, our technique with morphisms on a three letter alphabet is more natural and is based on the nice properties of return words and derivated sequences. 
This is why we have chosen to write the paper to develop the whole theory of substitutive Rote sequences. 

\begin{rmrk} 
The Rote sequence ${\bf r}$ from the beginning of this section is connected to the fixed point ${\bf u}$ of the Sturmian morphism $\psi = \varphi_{\beta}\varphi_b E$. 
Of course, ${\bf u}$  is fixed also by the morphisms $\psi^2= \varphi_{\beta bb\beta }$ which we have studied in Remark \ref{exceptional}. 
It can be shown that the Rote sequence $\vv$ associated with this $\bf u$ is exceptional among all Rote sequences associated with standard Sturmian sequences since it has only two distinct derivated sequences to its prefixes. The other Rote sequences have at least three derivated sequences.
\end{rmrk}
\section{Comments}
• In this paper, we have studied only the Rote sequences whose associated Sturmian sequences are standard. 
By definition, the intercept of a standard Sturmian sequence $\uu$ is equal to $1-\alpha$, where $\alpha$ is the density of  the letter $0$ in $\uu$. 
For such a sequence we have used its S-adic representation $\zz$ consisting of the morphisms $\varphi_b:0\to 0, 1\to 01$ and $\varphi_\beta: 0\to 10, 1\to 1$.
In particular, we have used the result from \cite{KlMePeSt18} which says that each suffix of $\zz$ represents a derivated sequence of $\uu$.

In \cite{Dek17} M. Dekking studied properties of some submonoids of the Sturmian monoid.  
In particular, he considered the submonoid (we kept  his notation) $\mathcal{M}_{3,8}$ generated by two morphisms $\psi_3: 0\to 0, 1\to 01$ and $\psi_8:0\to 01, 1\to 1$.  
Theorem 3 from \cite{Dek17} says that any fixed point $\uu$ of a primitive morphism from $\mathcal{M}_{3,8}$ is a Sturmian sequence with the intercept $0$.  
Obviously, this $\uu$ has an $S$-adic representation
$\zz$ consisting of the morphisms $\psi_3$ and $\psi_8$.
But unlike the case of standard Sturmian sequences, in this case only some of suffixes of $\zz$ represent derivated sequences to prefixes of $\uu$, see \cite{KlMePeSt18}. 
It would be interesting to know how this fact influences the set of derivated sequences of a Rote sequence associated with a Sturmian sequence with the intercept $0$.

• The definition of derivated sequences of $\uu$ as introduced in \cite{Dur98} takes into account only the prefixes of $\uu$. 
Recently, Yu-Ke Huang and Zhi-Ying Wen in \cite{HuWe17} have considered also the derivated sequences of $\uu$ to non-prefix factors of $\uu$.
Recall that if $\uu$ is a fixed point of a primitive morphism, then by Durand's result from \cite{Dur98}, any derivated sequence to a prefix of $\uu$ is fixed by a primitive morphism as well. 
However, a derivated sequence to a non-prefix factor of $\uu$ need not to be fixed by a non-identical morphism at all.

Huang and Wen study the period-doubling sequence ${\bf p}$, i.e. the sequence fixed by the morphism $0\mapsto 11, 1\mapsto 10$.   
They show that there exist two sequences $\Theta_1$ and $\Theta_2$ such that any derivated sequence ${\bf d}$ to a factor of $\uu$ is equal to $\Theta_1$ or $\Theta_2$.
Moreover, any derivated sequence ${\bf d}'$ to a factor of ${\bf d}$ is equal to $\Theta_1$ or $\Theta_2$ and any  derivated sequence ${\bf d}''$ to a factor of  ${\bf d}'$ is equal to $\Theta_1$ or $\Theta_2$, etc.
They called this property Reflexivity. 
It may be interesting to look for sequences with Reflexivity among Rote or Sturmian sequences. 

Let us note that the period-doubling sequence $\bf p$ and the Thue-Morse sequence ${\bf t}$, i.e. the sequence fixed by the morphism $0\mapsto 01, 1\mapsto 10$, are linked with the same mapping $\S$ which associates the Rote and Sturmian sequences (see Definition \ref{S}): $\S({\bf t}) = {\bf p}$.

• By Corollary \ref{Thm_Number}, a Rote sequence $\vv$ associated with a fixed point $\uu$ of a primitive standard Sturmian morphism $\varphi_z$ has at most $3|z|$ distinct derivated sequences.
On the other hand, if $\varphi_z$ is not a power of any other morphism, then the Sturmian sequence $\uu$ has exactly $|z|$ distinct derivated sequences (see \cite{KlMePeSt18}) and thus by Corollary \ref{Thm_DetermineDerivatedWord}, the Rote sequence $\vv$ has at least $|z|$ distinct derivated sequences. 
In all examples for which we have listed the derivated sequences of the fixed point of such a $\varphi_z$, the actual number of derivated sequences was  $|z|$, $2|z|$ or $3|z|$.
We do not know whether some other values can also appear.

\section*{Acknowledgement}  
The research received funding from the Ministry of Education, Youth and Sports of the Czech Republic through the project no. CZ.02.1.01/0.0/0.0/16\_019/0000778, from the Czech Technical University in Prague through the project SGS17/193/OHK4/3T/14, and from the French Agence Nationale de la Recherche (ANR) through the GlobNets

\end{document}